\documentclass{amsart}
\usepackage{cases}
\usepackage{bbm}
\usepackage{stmaryrd}
\usepackage{cases}
\usepackage{mathrsfs}
\usepackage{amsfonts}
\usepackage{amsmath}
\usepackage{amsthm}
\usepackage{empheq}
\usepackage{enumitem}
\newtheorem{theorem}{Theorem}[section]
\newtheorem {lemma}[theorem]{Lemma}
\newtheorem{corollary}[theorem]{Corollary}
\newtheorem{proposition}[theorem]{Proposition}
\theoremstyle{definition}

\theoremstyle{remark}
\newtheorem {remark}[theorem]{Remark}
\newtheorem{conditions}[theorem]{Conditions}
\newtheorem{example}[theorem]{Examples }
\newcommand{\dif}{\mathrm{d}}
\newcommand{\Hess}{\mathrm{Hess}}

\newcommand{\cc}{\mathrm{c}}

\newcommand{\e}{\mathrm{e}}

\newcommand{\graph}{\mathrm{graph}}
\newcommand{\tr}{\mathrm{tr}}

\def\ds{\displaystyle}
\def\vle{{\rm Vol}}

 \newcommand{\os}{\overset}

 \numberwithin{equation}{section}

\begin{document}

\title[MVPF by powers of  homogeneous curvature functions of degree one]
{Mixed volume preserving flow by powers of  homogeneous curvature functions of degree one}

  \author{Shunzi Guo}
\address{School of Mathematics and Statistics, Minnan Normal University, Zhangzhou, 363000, People's Republic of China \\
and
School of Mathematics \\Sichuan University\\Chengdu 610065, People's Republic of China }
\email{guoshunzi@yeah.net}


\renewcommand{\subjclassname}{%
  \textup{2010} Mathematics Subject Classification}



\begin{abstract}
This paper concerns the evolution of a closed hypersurface of dimension $n(\geq 2)$ in
the Euclidean space ${\mathbb{R}}^{n+1}$ under a mixed volume preserving flow.
The speed equals a power $\beta (\geq 1)$ of
 homogeneous, either convex or concave,
curvature functions of degree one plus a mixed volume preserving term, including
the case of powers of the mean curvature and of the Gauss curvature. The main result is that if the initial hypersurface
satisfies a suitable pinching condition, there exists a unique, smooth solution of the flow for all times,
and the evolving hypersurfaces converge exponentially to a round
sphere, enclosing the same mixed volume as the initial hypersurface.
This result covers and generalises the previous results for convex hypersurfaces
in the Euclidean space by McCoy \cite{McC05} and  Cabezas-Rivas and Sinestrari \cite{CS10}
 to more general curvature flows for convex hypersurfaces with similar curvature pinching condition.
\end{abstract}
\maketitle

\section{\bf Introduction}

\label{intro} Let $M^n$ be a smooth, compact oriented manifold of
dimension $n (\geq 2)$ without boundary, and
$X_{0}:M^{n}\rightarrow {\mathbb{R}}^{n+1}$ a smooth immersion of $M^n$ into the Euclidean space. We are interested in a
one-parameter family of smooth immersions: $X_{t}:M^{n}\rightarrow
{\mathbb{R}}^{n+1}$ evolving according to
\begin{equation}\label{mvpcf}
\left\{
\begin{array}{ll}
\frac{\partial}{\partial t}\mathrm{X}\left(p,t\right)
    = \{ \bar{\phi}(t)-\Phi ( F(\mathscr{W}\left(p,t\right))\}) \nu\left(p,t\right),& p\in M^{n},\\[2ex]
X(\cdot,0)=X_{0}(\cdot),
\end{array}\right.
\end{equation}
where $\nu\left(p,t\right)$ is the outer unit normal to
$M_{t}=X_{t}(M^{n})$ at the point $\mathrm{X}\left(p,t\right)$,
$\mathscr{W}_{-\nu}\left(p,t\right)=-\mathscr{W}_{\nu}\left(p,t\right)$
is the matrix of the Weingarten map on the tangent space $TM^{n}$
induced by $X_{t}$, $\Phi$ is a smooth supplementary function and
the $F(\mathscr{W})$ will be specified below.
The $ \bar{\phi}(t)$ is a mixed volume preserving global term,
and stands for the averaged $\Phi$ in the sense of the mixed volumes of a convex hypersurface $M_{t}$:
\begin{equation}\label{barPhi}
  \bar{\phi}_{m}(t)=\frac{\int_{M_{t}}E_{m+1}\Phi(F)\dif\mu_{t}}{\int_{M_{t}}E_{m+1}\dif\mu_{t}},
\end{equation}
for each $m=-1, 0, 1, \ldots, n-1$.
Here $\dif\mu_{t}$ denotes the surface area element of $M_{t}$, and
 $E_{m}$ is the $m$th elementary symmetric functions,
\begin{equation}\label{def Em}
E_{m}=\left\{ \begin{aligned}
         & 1 && m= 0 \\
         & \sum\limits_{1
\leq i_{1}< \cdots < i_{m} \leq n}\lambda_{i_{1}}\cdots
\lambda_{i_{m}} && m=1, \cdots, n.
                \end{aligned} \right.
\end{equation}
Obviously $E_{1}=H$ and $E_{n}=K$, where $H$ and $K$ denote the
mean curvature and the Gau\ss-Kronecker curvature respectively.
For each fixed $m$, as is clear from the presence of the global term $\bar{\phi}(t)$ in
equation \eqref{mvpcf}, the flow keeps the mixed volumes of $\Omega_{t}$,
where $\Omega_{t}$ is the solid $(n+1)$-region bounded by $M_{t}$.
\begin{equation} \label{def mixedvolume}
V_{n-m}(\Omega_{t})=\left\{ \begin{aligned}
         & \vle(\Omega_{t}) && m= -1 \\
                  &\left\{(n+1) {n \choose m}\right\}^{-1}\int_{M_{t}}E_{m}\dif\mu_{t}&&m= 0,1,\cdots,n-1
                          \end{aligned} \right.
                         \end{equation}
 constant. Here $\vle(\Omega_{t})$ denotes the volume of the domain enclosed by $M_{t}$.

We require that the $F(\mathscr{W})$ satisfies the following properties:
\begin{conditions} \label{Fconds}
\begin{enumerate}[label={(\roman*).}, ref={(\roman*)}]
  \item\label{item1} $F= f\circ \lambda$, where $\lambda$ is the map which takes a self-adjoint operator to its ordered eigenvalues, and $f$ is a smooth, symmetric function defined on an open symmetric cone $\Gamma\subseteq \Gamma_+$ containing $\{(c,\dots,c):\ c>0\}$, where $\Gamma_+=\{\lambda:\ \lambda_i>0,\ i=1,\dots,n\}$.
    \item\label{item2} $f$ is strictly monotone: $\frac{\partial f}{\partial \lambda_{i}} >0$ for each $i = 1, \ldots, n$, at each point of $\Gamma$.
  \item\label{item3} $f$ is homogeneous of degree one: $f\left( k \lambda \right) = k f\left( \lambda \right)$ for any $k >0$.
  \item\label{item4} $f$ is strictly positive on $\Gamma$ and normalised to have $f\left( 1, \ldots, 1 \right) = 1$.
  \item\label{item5} Either:
  \begin{itemize}
  \item\label{item6} $f$ is convex, or
    \item\label{item7} $f$ is concave and one of the following hold
    \begin{enumerate}
    \item \label{item8} $f$ approaches zero on the boundary of $\Gamma$,
    \item \label{item9} $f$ is inverse concave, that is, the function $\tilde{f}(x_{1}, \ldots, x_{n})=-f(x^{-1}_{1}, \ldots, x^{-1}_{n})$ is concave.
   \end{enumerate}
    \end{itemize}
\end{enumerate}
\end{conditions}

Examples of $F(\mathscr{W})$ fulfilling the conditions \ref{Fconds} include, apart from the normalization:
\begin{example}
 \begin{itemize}
\item The concave and inverse concave examples of curvature functions $F$: $\left(\frac{E_m}{E_l}\right)^{\frac{1}{m-l}}$, $n\geq m> l \geq 0$, or the power means $\left(\sum_{i=1}^n  \lambda_i^r \right)^{\frac{1}{r}}$ for $|r|\leq 1$ . For a proof of the inverse concavity of these functions see the proofs of \cite[Theorem 2.6, Theorem 2.7]{And07}.
\item The convex examples of curvature functions $F$,
for a proof that these curvature functions are convex, see \cite[p. 105]{Mit70}.:
\begin{enumerate}
\item the mean curvature $H = \sum_{i=1}^n  \lambda_i$,
\item the length of the second fundamental form $|A| = \sqrt{\sum_{i=1}^n \lambda_i^2}$ and the completely symmetric functions $\gamma_k = \left(\sum_{|\beta| = k}  \lambda^\beta\right)^{\frac{1}{k}}$, $1\leq k\leq n$,
where $\beta$ is a multiindex, $\beta \in \mathbb{N}^n$, and $ \lambda^\beta =  \lambda_1^{\beta_1}\cdot \lambda_2^{\beta_2}\ldots\cdot  \lambda_n^{\beta_n}$.
\end{enumerate}
    \end{itemize}
\end{example}
There exists a wide literature about cuvvature problems of the form \eqref{mvpcf}, for which
the speed $\Phi(F)= F$ satisfies Conditions \ref{Fconds} for hypersurfaces in different ambient spaces.
First, in the case of the volume-preserving mean curvature flow, Huisken \cite{Hui87} showed that convex
Euclidean hypersurfaces remain convex for all time and converge exponentially
fast to round spheres (the corresponding result for curves in the
Euclidean plane is due to Gage \cite{Ga86}),
 while Andrews \cite{And01}
extended this result to the smooth anisotropic mean curvature flow,
and McCoy showed similar results for the Euclidean surface area preserving
mean curvature flow \cite{McC03} and the mixed volume preserving
mean curvature flows \cite{McC04}. The volume-preserving flow has
been used to study constant mean curvature surfaces between parallel
planes \cite{At97,At03} and to find canonical foliations near
infinity in asymptotically flat spaces arising in general relativity
\cite{HY96} (Rigger \cite{Rig04} showed analogous results in the
asymptotically hyperbolic setting). If the initial hypersurface is
sufficiently close to a fixed Euclidean sphere (possibly
non-convex), Escher and Simonett \cite{ES98} proved that the flow
converges exponentially fast to a round sphere, a similar result for
average mean convex hypersurfaces with initially small traceless
second fundamental form is due to Li \cite{Li09}. For a general
ambient manifold,
Alikakos and Freire \cite{AF03} proved long time
existence and convergence to a constant mean curvature surface under
the hypotheses that the initial hypersurface is close to a small
geodesic sphere and that it satisfies some non-degenerate
conditions. Cabezas-Rivas and Miquel exported the Euclidean
results of \cite{At97,At03} to revolution hypersurfaces in a
rotationally symmetric space \cite{CR-M09}, and showed the same
results as Huisken \cite{Hui87} for a hyperbolic background space
\cite{CR-M07} by assuming the initial hypersurface is
horospherically convex. On the other hand, there are few results on speeds different from
the mean curvature: McCoy \cite{McC05} proved the convergence to a
sphere for a large class of function $F$ homogeneous of degree one
(including the case $\Phi(F)=E_{m}^{\beta}$ with $m\beta=1$), Makowski
showed the volume preserving curvature flow in Lorentzian manifolds for $F$ as
a function with homogeneous of degree one exponential converges to a
hypersurface of constant $F$-curvature \cite{Mak13} (moreover,
stability properties and foliations of such a hypersurface  was also
examined). In 2010, Cabezas-Rivas and Sinestrari \cite{CS10} studied
the deformation of hypersurfaces with initial pinching condition \eqref {concave pinched}
in ${\mathbb{R}}^{n+1}$ by a
speed with the form $\Phi(F)=E_{m}^{\beta}$ for some power $\beta \geq 1/m$. In
this way $\Phi(F)$ is a homogeneous function of the curvatures with a
degree $m\beta \geq 1$. The results of \cite{CS10} do not closely relate to the
ambient space. Recently, The author,
together with Li and Wu \cite{GuoLW16},
achieved such extension of the above Theorem \ref{main result for powerF flow} of
Cabezas-Rivas and Sinestrari \cite{CS10} to convex hypersurfaces with similiar pinching condition in the hyperbolic case.

Inspired by these results of \cite{CS10, GLW16, McC05},
in this paper we replace the speed function $F(\lambda)$ in \cite{McC05} by
a power of the  $F(\lambda)$, namely we study the long time existence and the convergence of the flow \eqref{mvpcf}
with the speed $\Phi$ given by a power of the curvature function, namely
\begin{equation} \label{def_Phi}
\Phi(F)= F^\beta
\end{equation}
for some $\beta \geq 1$.
In this way $\Phi$ is a homogeneous function of the curvatures with a degree $\beta\geq 1$.
We call the flow \eqref{mvpcf} with this choice of speed {\it the mixed volume preserving $F^\beta$-flow}.
Our analysis will focused on the existence of this deformation of convex hypersurfaces and its limiting behavior.
It is well-known that preserving convexity of hypersurfaces
is a key point to analyze the convergence of the flow.
However, we face the first difficulty:
Due to higher homogeneity of the speed,
we cannot directly apply the maximum principle of Hamilton \cite{Ham82} to the evolution of the principal curvature
to show that convexity of hypersurfaces is preserved.
Thus, we have to search some ``nice" geometric quantities of the principal curvature
for which the maximum principle can be applied to the evolution.
So the next problem we meet is:
what is the suitable geometric quantity for the mixed volume preserving $F^\beta$-flow.
Meanwhile we observe that most of the literature in the investigation of evolution equations,
requires an additional assumption of a suitable geometric quantity on the initial hypersurface
for example, of $K/H^{n}$ used for the gauss curvature flow \cite{Cho85},
the power mean curvature flow \cite{Sch05,GLW16}
and the volume preserving $E_{m}^\beta$-flow \cite{CS10,GuoLW16},
$R/H^{2}$ used for the flow by the square root of the scalar curvature \cite{Cho87},
and $K/F^{n}$ used for the mixed volume preserving $F$-flow  \cite{McC05}.
We realized that the two scalars $K/F^{n}$ for convex $F$ and $K/H^{n}$ for concave $F$
should be good candidates for our flow.
In Section \ref{Preserving pinching for powerF flow},
we shall succeed in estimating the two scalar.
In this paper, by imposing additional assumptions on the two geometric quantities mentioned above on the initial hypersurface,
we shall prove
\begin{theorem}\label{main result for powerF flow} Suppose $F$ satisfies Conditions \ref{Fconds},
and for $\beta \geq 1$ there exists a positive
constant $\mathcal{C}=\mathcal{C}(n, \beta )$ such that
the following holds: If the initial hypersurface of
${\mathbb{R}}^{n+1}$  is pinched at every point in the case that
\begin{enumerate}[label={(\roman*).}, ref={(\roman*)}]
\item {for convex $F$},\begin{equation}\label{convex pinched}
 K(p)> \mathcal{C}F^n>0,
\end{equation}
\item {for concave $F$},
\begin{equation}\label{concave pinched}
 K(p)> \mathcal{C}H^n>0,
\end{equation}
\end{enumerate}
then the flow \eqref{mvpcf}-\eqref{barPhi} with $\Phi$ given by
\eqref{def_Phi} has a unique and smooth solution for all times,
inequality \eqref{convex pinched} or \eqref{concave pinched} remains true everywhere on the
evolving hypersurfaces $M_{t}$ for all $t>0$ and the $M_{t}$'s
converge, as $t\rightarrow \infty$, exponentially in the $C^{\infty}$-topology,
to a round sphere enclosing the same value of $V_{n-m}$ as $M_{0}$.
\end{theorem}
\begin {remark}
Since the main result of Theorem \ref{main result for powerF flow} cover
the previously known results in $\mathbb{R}^{n+1}$: of the volume-preserving mean curvature flow \cite{Hui87},
the surface area preserving mean curvature flow \cite{McC03}, the mixed volume preserving mean curvature flows \cite{McC04}
and the volume-preserving $E_{m}^{\beta}$-flow \cite{CS10}.
Nevertheless, our results are a significant extension of those results in this direction.
\end {remark}

Our work is mainly motivated by the approaches in the papers \cite{CS10,McC05}, we make
modifications to consider our problem for the different speed \eqref{def_Phi}. The
rest of the paper is organized as follows:
Section \ref{notatons for the powerF flow} first gives
some useful preliminary results employed in the remainder of the
paper.  Section \ref{shorttime for the powerF flow} contains details of the short time existence of
the flow \eqref{mvpcf}-\eqref{barPhi} and the induced evolution
equations of some important geometric quantities. In
 Section \ref{Preserving pinching for powerF flow} applying the maximum principle to the evolution equation
of the two quantities $K/F^{n}$ for convex $F$ and $K/H^{n}$ for concave $F$ gives that if the
initial hypersurface is pinched good enough then this is preserved
for $t>0$ as long as the flow \eqref{mvpcf}-\eqref{barPhi}
exists. This is a fundamental step in our procedure as in most of
the literature quoted above. Furthermore,
Section \ref{up bound for the powerF flow speed} proves the
uniform bound of the speed by following a method which was
firstly used by Tso \cite{Tso}. Using more sophisticated results for
fully nonlinear elliptic and parabolic partial differential equations,
Section \ref{longtime for the powerF flow}  obtains uniform bounds on all derivatives of
the curvature and proves long time existence of the flow \eqref{mvpcf}-\eqref{barPhi}.
Finally  Section \ref{convergence for the powerF flow},
following the ideas in \cite{Smo98} and \cite{And12},
obtains the lower speed bounds, which will then allow us to prove that these
evolving hypersurfaces converge to a sphere of
$\mathbb{R}^{n+1}$ smoothly and exponentially.

\section{\bf Notations and preliminary results}\label{notatons for the powerF flow}

From now on, we will follow similar notation used in \cite{CS10}.
In particular, in local coordinates $\{x^{i}\}$, $1\leq i \leq n$, near $p\in
M^{n}$ and $\{y^{\alpha}\}$, $0 \leq \alpha, \beta \leq n$, near
$F(p)\in \mathbb{R}^{n+1}$. Denote by a bar all
quantities on $\mathbb{R}^{n+1}$, for example by
$\bar{g}=\{\bar{g}_{\alpha \beta}\}$ the metric, by
$\bar{g}^{-1}=\{\bar{g}^{\alpha \beta}\}$ the inverse of the metric,
 by $\bar{\nabla}$ the covariant derivative, by $\bar{\Delta}$ the rough Laplacian.
 Components are sometimes taken with respect to the
tangent vector fields $\partial_{\alpha}(=\frac{\partial}{\partial
y^{\alpha}})$ associated with a local coordinate $\{y^{\alpha}\}$
and sometimes with respect to a moving orthonormal frame
$e_{\alpha}$, where $\bar{g}(e_{\alpha},
e_{\beta})=\delta_{\alpha\beta}$. The corresponding geometric
quantities on $M^{n}$ will be denoted by $g$ (the induced metric),
$g^{-1}, \nabla, \Delta, \partial_{i}$ and $e_{i}$, etc.. Then further important quantities are
the second fundamental form $A(p)=\{h_{ij}\}$ and the Weingarten map
$\mathscr{W}=\{g^{ik}h_{kj}\}=\{h^{i}_{j}\}$ as a symmetric operator
and a self-adjoint operator respectively. The eigenvalues
$\lambda_{1}(p)\leq \cdots \leq \lambda_{n}(p)$ of $\mathscr{W}$ are
called the principal curvatures of ${X(M^{n})}$ at
${X(p)}$. The mean curvature is given by
\[H:=\tr_{g}{\mathscr{W}}=h^{i}_{i}=\sum_{i=1}^{n}\lambda_{i},\]
the squared norm of the second fundamental form by
$$\bigl|A\bigr|^{2}:=\tr_{g}({\mathscr{W}^{t}\mathscr{W}})=h^{i}_{j}h^{j}_{i}=h^{ij}h_{ij}=\sum_{i=1}^{n}\lambda^{2}_{i},$$
and Gau\ss-Kronecker curvature by
$$K:=\det(\mathscr{W})=\det\{h^{i}_{j}\}=\frac{\det\{h_{ij}\}}{\det\{g_{ij}\}}=\prod_{i=1}^{n}\lambda_{i}.$$
More generally, the $m$th elementary symmetric functions $E_{m}$ are
given by
$$E_{m}(\lambda)=\sum\limits_{1 \leq
i_{1}<  \cdots < i_{r} \leq n}\lambda_{i_{1}}\cdots
\lambda_{i_{r}},\ \ {\hbox
{for}}\  \lambda = (\lambda_{1},\ldots,\lambda_{n})\in
\mathbb{R}^{n},$$
we also set $E_{m}=1$ if $m = 0$.
Thus the $m$th mean curvatures $E_{m}$ are given by
\eqref{def_Phi}.
In addition we define
\[
\Gamma_{m}:=\{\lambda\in\mathbb{R}^{n}|E_{1}(\lambda)>0, E_{2}(\lambda)>0, \cdots, E_{m}(\lambda)>0\}
\]
and
the quotients
\[
{Q_{m}}=\frac{E_{m}(\lambda)}{E_{m-1}(\lambda)},\qquad \forall \lambda \in\Gamma_{m}.
\]
Since $E_{m}$ is homogeneous of degree $m$, the
speed $F$ is homogeneous of degree $\beta$ in the curvatures
$\lambda_{i}$. Denote the vector $(\lambda_{1},\dots, \lambda_{n})$
of $\mathbb{R}^{n}$ by $\lambda$ and the positive cone by
$\Gamma_{+}\subset \mathbb{R}^{n}$, i.e.
\[
\Gamma_{+}=\{\lambda=(\lambda_{1},\dots, \lambda_{n}):\lambda_{i}>0,
\forall\ i\}.
\]
It is clear that $H$, $K$, $E_{m}$, $Q_{m}$, $F$ may be viewed as functions
of $\lambda$, or as functions of $A$, or as functions of
$\mathscr{W}$, or also functions of space and time on $M_{t}$.
Throughout this paper we sum over repeated indices from $1$ to $n$ unless otherwise indicated.
In computations on the hypersurface $M_{t}$, raised indices indicate contraction with the metric.

We will denote by $ \dot{F}^{kl}$ the matrix of the first partial derivatives of $F$ with respect to the components of its argument:
$$\left. \frac{\partial}{\partial s} F\left( A+sB\right) \right|_{s=0} = \dot F^{ij} \left( A\right) B_{ij} \mbox{.}$$
Furthermore we denote the second partial derivatives of $F$ by
$$\left. \frac{\partial^{2}}{\partial s^{2}} F\left( A+sB\right) \right|_{s=0} = \ddot F^{kl, rs} \left( A\right) B_{kl} B_{rs} \mbox{.}$$
If we do not indicate explicitly where derivatives of $F$ and of $f$ are evaluated then they are evaluated at $\mathscr{W}$ and $\lambda\left(\mathscr{W} \right)$ respectively.
We will further use the shortened notation $\dot{f}^{i}=\frac{\partial f}{\partial \lambda_{i}}$
and $\ddot{f}^{ij}=\frac{\partial^{2} f}{\partial \lambda_{i}\partial \lambda_{j}}$  where appropriate.
We will use similar notation and conventions for other functions of matrices when we differentiate them.

Now we state the following well-known inequalities for general curvature functions $F$:
\begin {lemma}
\label{FHInequalities}
For any concave \,(convex) $F$  fulfilling the conditions \ref{Fconds},
	\begin{equation*}
	F(\lambda) \leq\, ( \geq)\, \frac{F(1, \ldots, 1)}{n} H(\lambda).
	\end{equation*}
	and
	\begin{equation*}
		\sum_{i=1}^n f_i(\lambda) =\tr(F^{kl})\geq \,( \leq)\, 1,
	\end{equation*}
	where $\lambda = (\lambda_k) \in \Gamma_+$.
\end {lemma}
\begin{proof}
	See \cite[Lemma 2.2.19, Lemma 2.2.20]{Ger06}. Also see \cite{Urb91} for concave $F$.
\end{proof}

We need the following relation between the first derivatives of $F$ and the first derivatives of $f$,
 of \cite[Corollary 3.2]{McC05}.

\begin {lemma}
\label{relation of dotF and dotf}
If $\mathscr{W}$ has distinct eigenvalues $\lambda_k$.
then $F$  is  concave \,(convex) at $\mathscr{W}$ if and only if $f$  is  concave \,(convex) at $\lambda(\mathscr{W})$
and for all $k\neq l$
	\begin{equation*}
	\frac{\dot{f}_{k}-\dot{f}_{l}}{\lambda_k-\lambda_l}\leq\,(\geq) \,0.
	\end{equation*}
	\end {lemma}

We also need the following inequalities for $F$.
 See also \cite[Corollay 3.3]{McC05} for a derivation.
\begin {lemma}
\label{FHInequalities2}
For any concave \,(convex) $F$ at $\mathscr{W}$ fulfilling the conditions \ref{Fconds},
	\begin{equation*}
	H\tr_{\dot F}( A\mathscr{W} ) - F |A|^{2} \leq 0\, ( \geq 0).
	\end{equation*}
	\end {lemma}

In our analysis of the flow equations we will use several geometric estimates for a compact, strictly convex hypersurface
with suitably pinched curvatures, which implies suitably pinching relation between outer radius (circumradius) and inner radius (inradius)
for enclosed region by the hypersurface.
The first of these appears in
(see also \cite{McC05}).
\begin {lemma}
\label{pinched curvatures estimates}
If a smooth, compact, uniformly convex manifold $M^n$ satisfies everywhere the pointwise curvature pinching estimate
$$\lambda_n \leq C_1 \lambda_1$$
for some constant $C_1 < \infty$, then the circumradius $\rho_{+}$ of $M^n$  satisfies
$$\rho_{+} \leq C_2 \, \rho_{-}$$
where $C_2= \big(\frac{n+2}{\sqrt{2}}\big)C_1$ and  $\rho_{-}$ denotes the inradius of $M^n$.
	\end {lemma}

The ratio of circumradius to inradius of $M^n$ gives estimates for the inradius and circumradius in terms of the mixed volumes of $M^n$.
\begin {lemma}
\label{radius estimates}
Under the assumptions of Lemma \ref{pinched curvatures estimates}, for any $m=1,\cdots,n+1$,
$$
\rho_-\geq \frac{1}{C_2} \left(\frac{V_m}{\omega_{n+1}}\right)^{\frac1m}
\quad\text{and}
\quad
\rho_+\leq {C_2} \left(\frac{V_m}{\omega_{n+1}}\right)^{\frac1m}.
$$
where $\omega_{n+1}$ dentotes the volume of the $(n+1)$-dimensional unit ball.
	\end {lemma}
\begin{proof}
We can argue exactly as in \cite[Corollary 3.6]{McC05}).
\end{proof}


 We note some important properties of $E_{m}$ (see \cite{CS10} for a
 simple derivation).
 \begin {lemma}\label{property Em}
 Let $1\leq m \leq n$ be fixed.
 \begin{itemize}
 \item[i)] The $m$th roots $E_{m}^{1/m}$ are
 concave in $\Gamma_{+}$.
 \item[ii)] For all $i $, $\frac{\partial E_{m}}{\partial \lambda_{i}}(\lambda)>0$,
  where $\lambda \in \Gamma_{+}$.
  \item[iii)] ${E}_m^{1/m} \leq \ds\frac{H}{n}$; equivalently, ${E}_m\leq \left(\ds\frac{H}{n}\right)^{m \beta}$.
 \end{itemize}
 \end {lemma}

 The following algebraic property proved by Schulze in \cite [Lemma\,2.5]{Sch06},
 will be needed in the later sections.
\begin {lemma}\label{important inequality}
For any $\varepsilon> 0$ assume that $\lambda_{i} \geq \varepsilon H
> 0$, $i= 1, \ldots, n$, at some point of an $n$-dimensional
hypersurface. Then at the same point there exists a $\delta=
\delta(\varepsilon,n)> 0$ such that
\[
\frac{n\big|A\big|^{2}-H^{2}}{H^{2}}\geq \delta \left(
\frac{1}{n^n}-\frac{K}{H^n}\right).
\]
\end {lemma}

In our analysis we need some a priori estimates on the H\"older
norms of the solutions to elliptic and parabolic partial
differential equations in Euclidean spaces. We recall that, in the
case of a function depending on space and time, there is a suitable
definition of H\"older norm which is adapted to the purposes of
parabolic equations (see e.g. \cite{Lieb}). In addition to the
standard Schauder estimates for linear equations, we use in the
paper some more recent results which are collected here.  The
estimates below hold for suitable classes of weak solutions; for the
sake of simplicity, we state them in the case of a smooth classical
solution, which is enough for our purposes.

Given $r>0$, we denote by $B_r$ the ball of radius $r>0$ in
$\mathbb{R}^n$ centered at the origin. First we recall a well known
result due to Krylov and Safonov, which applies to linear parabolic
equations of the form
\begin{equation} \label{ec_par} \left(a^{ij}(x, t) D_{i} D_j  +
b^i(x, t) D_i + c(x, t) - \frac{\partial}{\partial t}\right) u = f
\end{equation}
in $B_r \times [0,T]$, for some $T>0$. We assume that $a^{ij} =
a^{ji}$ and that $a^{ij}$ is uniformly elliptic; that is, there exist two
constants $\lambda, \, \Lambda > 0$ such that
\begin{equation}
\label{SP} \lambda |v|^2 \leq a^{ij}(x,t)v_iv_j \leq \Lambda |v|^2
\end{equation}
for all $v \in \mathbb{R}^n$ and all $(x, t) \in B_r \times [0,T]$.
Then the following estimate holds \cite[Theorem 4.3]{KS}:

\begin{theorem} \label{krylov-safonov} Let $u \in C^2(B_r \times [0,T])$ be a solution of
\eqref{ec_par}, where the coefficients are measurable, satisfy
\eqref{SP} and
$$|b^i|, \, |c|  \leq \lambda_1 \qquad \text{for all } i = 1, \ldots, n,$$
for some $\lambda_1>0$. Then, for any  $0<r'<r$ and any $0<\delta <T$ we
have
$$\|u\|_{C^{\alpha}(B_{r '} \times [\delta, T])} \leq C\left(\|u\|_{C(B_r \times [0,T])} + \|f\|_{L_\infty(B_r \times [0,T])}\right)$$
for some constants $C > 0$ and $\alpha \in (0,1)$ depending on $n$,
$\lambda$, $\Lambda$, $\lambda_1$, $r$, $r'$ and $\delta$. \end{theorem}

Next we quote a result for fully nonlinear elliptic equations, which
is due to Caffarelli. We consider the equation
\begin{equation}
\label{FNL} F(D^2u(x),x)=f(x), \qquad x \in B_r.
\end{equation} Here
$F:{\mathcal S} \times B_r \to \mathbb{R}$, where $\mathcal S$ is
the set of the symmetric $n \times n$ matrices. The nonlinear
operator $F$ is called uniformly elliptic if there exist $\Lambda \geq \lambda
>0$ such that
\begin{equation} \label{elliptic condition} \lambda ||B|| \leq F(A+B,x) - F(A,x)
\leq \Lambda ||B||
\end{equation} for any $x \in B_r$ and any pair $A,B \in \mathcal S$ such that $B$
is nonnegative definite.

\begin{theorem}  \label{Caffarelli} Let $u \in C^2(B_r)$ be a solution of
\eqref{FNL}, where $F$ is continuous and satisfies \eqref{elliptic condition}.
Suppose in addition that $F$ is concave with respect to $D^2u$ for
any $x \in B_r$. Then there exists $\bar \alpha \in \,(0,1)$ with
the following property: if, for some $\lambda_2>0$ and $\alpha \in
\,(0,\bar \alpha)\,$, we have that $f \in C^\alpha(\Omega)$ and that
$$
F(A,x) - F(A,y) \leq \lambda_2 |x-y|^{\alpha} (||A||+1), \qquad x,y \in
B_r, \  A \in \mathcal S,
$$
then, for any $0<r'<r$, we have the estimate
$$\|u\|_{C^{2+\alpha}(B_{r'})} \leq C(||u||_{C(B_r)}+||f||_{C^\alpha(B_r)}+1)$$
where $C > 0$ only depends on $n$, $\lambda$, $\Lambda$, $\lambda_2$, $r$
and $r'$. \end{theorem}

The above result follows from Theorem 3 in \cite{Caf89} (see also
Theorem 8.1 in \cite{CC} and the remarks thereafter). It
generalizes, by a perturbation method, a previous estimate, due to
Evans and Krylov, about equations with concave dependence on the
hessian. In contrast with Evans-Krylov result (see e.g. inequality
(17.42) in \cite{GT}), Theorem \ref{krylov-safonov} gives an
estimate in terms of the $C^\alpha$-norm of $f$ rather than the
$C^2$-norm, and this is essential for our purposes.

Finally, we recall an interior H\"older estimate, due to Di
Benedetto and Friedman \cite[Theorem 1.3]{DF85}, for solutions of
the degenerate parabolic equation
\begin{equation} \label{DF85}
\frac{\partial v}{\partial t} - D_i\left(a^{ij}(x, t, D v) D_j
v^d\right) = f(x,t,v, D v),
\end{equation} being $d > 1$.

\begin{theorem} \label{C2alpha estmate due to caf} Let $v \in C^2(B_r \times [0,T])$ be a nonnegative
solution of \eqref{DF85}, where $a^{ij}$ satisfies \eqref{SP}. Let
$c_1, c_2, N>0$ be such that
$$|f(x,t,v, D v)|\leq c_1 |D v^d| + c_2,$$
and
 $$\sup_{0 < t < T} ||v(\, \cdot \,, t)||_{L^2(B_r)}^2 + \|D v^d\|_{L^2(B_r \times [0,T])}^2 \leq N.$$ Then for any $0<\delta<T$ and $0<r'<r$, we have
$$\|v\|_{C^{\alpha}\left(B_{r'} \times [\delta, T]\right)}  \leq C,$$
for suitable $C>0, \alpha \in (0, 1)$ depending only on $n, N,
\lambda, \Lambda, \delta, c_1, c_2, r$ and $r'$. \end{theorem}

\section{\bf Short time existence and evolution equations}\label{shorttime for the powerF flow}

This section first considers short time existence for the initial
value problem \eqref{mvpcf}.
\begin{theorem}
Let $\mathrm{X_{0}}:M^{n}\rightarrow \mathbb{R}^{n+1}$ be
a smooth closed hypersurface with $F>0$ everywhere. Then there exists a unique smooth solution
$\mathrm{X_{t}}$ of problem  \eqref{mvpcf}-\eqref{barPhi} with $\Phi$ given by
\eqref{def_Phi}, defined on some time
interval $[0, T )$, with $T > 0$.
\end{theorem}
\begin{proof}
We can argue exactly as in \cite[Theorem 3.1]{CS10}; although the
assumptions on the initial hypersurface and the speeds in that paper are different, the proof applies to our case as
well.
\end{proof}

Proceeding now exactly as in \cite{Hui84, Ger06} from the basic equation \eqref{mvpcf} with $\Phi$ given by \eqref{def_Phi}
we can easily compute how the following geometric quantities  on $M_{t}$ evolve.
\begin{proposition}
On any solution $M_{t}$ of \eqref{mvpcf}-\eqref{barPhi} with $\Phi$ given by \eqref{def_Phi} the following hold:
\allowdisplaybreaks
\begin{align}
\label{evmetric for the powerF}
&\partial_{t}g=2(\bar{\phi}-\Phi)A,\\
&\partial_{t}g^{-1}=-2(\bar{\phi}-\Phi)g^{-1}\mathscr{W},\notag\\
&\partial_{t}\nu=\mathrm{X_{*}}(\nabla \Phi),\notag\\
\notag&\partial_{t}(\dif \mu_{t})=(\bar{\phi}-\Phi)H\dif \mu_{t},\\
\notag&\partial_{t}A=\Delta_{\dot{\Phi}} A
  +  \ddot{\Phi}(\nabla_{_{\! \ds \cdot}} \mathscr{W} ,\nabla_{_{\! \ds \cdot}} \mathscr{W})
  + \tr_{\dot \Phi}(A\mathscr{W} ) \, A + \big[\bar{\phi} - ( \beta + 1) \Phi\big] A \mathscr{W},\\
 \label{evweigar for the powerF}
&\partial_{t}\mathscr{W}=\Delta_{\dot{\Phi}} \mathscr{W}
  +  \ddot{\Phi}(\nabla_{_{\! \ds \cdot}} \mathscr{W} ,\nabla^{_{\! \ds \cdot}} \mathscr{W})
  +\tr_{\dot \Phi}(A\mathscr{W} ) \, \mathscr{W} - \big[\bar{\phi}+ ( \beta - 1) \Phi\big] \mathscr{W}^{2}.
\end{align}
In addition, the position vector field $\mathfrak{X}(\cdot, t) =X(\cdot, t) - x_{0}$  (with origin $x_{0}$) on $M_t$ evolves according to
\begin{equation}
\label{ev support for the powerF}
\partial_{t} \langle\mathfrak{X},  \nu\rangle= \Delta_{\dot \Phi} \langle\mathfrak{X},  \nu\rangle+ \tr_{\dot \Phi} (\mathscr{W} A) \langle\mathfrak{X},  \nu\rangle+ \big(\bar{\phi} - (\beta + 1) \Phi\big).
\end{equation}
\end{proposition}

We are going to prove the claimed the mixed volume preservation property of the flow \eqref{mvpcf} with $\bar{\phi}_{m}(t)$,
 $m=-1, \ldots, n-1$.
\begin {lemma}\label{mixed volume preservation}
Any flow of the form \eqref{mvpcf} with $\bar{\phi}=\bar{\phi}_{m}(t)$ given by \eqref{barPhi} preserves the mixed volume $V_{n-m}$.
\begin{proof}
By the definition \eqref{def mixedvolume} of the mixed volume,
we now recall from \cite[Lemma 4.3]{McC05},
 along any flow of the form \eqref{mvpcf},
\begin{equation*}
\frac{\dif}{\dif t}\int_{M_{t}}E_{m}\dif\mu_{t}=\left\{ \begin{aligned}
         & 0 && m= n \\
                  &(m+1) \int_{M_{t}} \left(\bar{\phi}-\Phi\right)E_{m+1}\dif\mu_{t}&&m= 0,1,\cdots,n-1
                          \end{aligned} \right.
\end{equation*}
Hence in view of that $\bar{\phi}=\bar{\phi}_{m}(t)$ given by \eqref{barPhi},
we get that for $m= 0,1,\cdots,n-1$
\begin{equation*}\frac{\dif}{\dif t}V_{n-m}=(m+1)\left\{(n+1) {n \choose m}\right\}^{-1}
        \int_{M_{t}} \left(\bar{\phi}-\Phi\right)E_{m+1}\dif\mu_{t}=0.
\end{equation*}

\end{proof}
\end {lemma}


In the next theorem, we derive the evolution of any homogeneous
function of the Weingarten map $\mathscr{W}$ defined on an evolving
hypersurface $M_{t}$ of $\mathbb{R}^{n+1}$ under the flow
\eqref{mvpcf}.
\begin{theorem}\label{evGD for the powerF}
If $G$ is a homogeneous function of the Weingarten map $\mathscr{W}$
of degree $\alpha$ , then the evolution equation of $G$ under the
flow \eqref{mvpcf} in $\mathbb{R}^{n+1}$ is the
following
\begin{align*}
\partial_{t}G=&\,\Delta_{\dot{\Phi}} G
-\dot{\Phi}\ddot{G}(\nabla_{_{\! \ds \cdot}} \mathscr{W} ,\nabla^{_{\!
\ds \cdot}} \mathscr{W})
  + \dot{G} \ddot{\Phi}(\nabla_{_{\! \ds \cdot}} \mathscr{W} ,\nabla^{_{\! \ds \cdot}} \mathscr{W})
  +\alpha \tr_{\dot \Phi}(A\mathscr{W} ) \, G\\
 &- \big[\bar{\phi}+ (\beta - 1) \Phi\big] \dot{G}\mathscr{W}^{2}.
\end{align*}
\end{theorem}
\begin{proof}
The definition of $\dot{G}$ and $\ddot{G}$ allow us to write
$\Hess_{\nabla}G$ as follows:
\begin{align*}
\Hess_{\nabla}G =\dot{G}\,\Hess_{\nabla}\mathscr{W} +
\ddot{G}(\nabla_{_{\! \ds \cdot}}\mathscr{W} ,\nabla^{_{\! \ds \cdot}}\mathscr{W}),
\end{align*}
which gives
\begin{align*}
\Delta_{\dot{\Phi}} G =\dot{\Phi}g^{-1}\,\Hess_{\nabla} G=
\dot{G}\,\Delta_{\dot{\Phi}} \mathscr{W}
  +  \dot{\Phi}\ddot{G}(\nabla_{_{\! \ds \cdot}} \mathscr{W} ,\nabla^{_{\! \ds \cdot}} \mathcal
  {W}).
\end{align*}
Therefore, by \eqref{evweigar for the powerF}
\begin{align*}
\partial_{t}G= &\,\dot{G} \partial_{t}\mathscr{W}\\
=&\,\dot{G}\Delta_{\dot{\Phi}} \mathscr{W}
  +  \dot{G}\ddot{\Phi}(\nabla_{_{\! \ds \cdot}} \mathscr{W} ,\nabla^{_{\! \ds \cdot}} \mathscr{W})
  +\tr_{\dot \Phi}(A\mathscr{W} ) \, \dot{G}\mathscr{W}- \big[\bar{\phi}+ ( \beta - 1) \Phi\big]\dot{G} \mathscr{W}^{2}\\
=&\,\Delta_{\dot{\Phi}} G -\dot{\Phi}\ddot{G}(\nabla_{_{\! \ds \cdot}}
\mathscr{W} ,\nabla^{_{\! \ds \cdot}} \mathscr{W})
  + \dot{G} \ddot{\Phi}(\nabla_{_{\! \ds \cdot}} \mathscr{W} ,\nabla^{_{\! \ds \cdot}} \mathscr{W})\\
&  +\alpha \tr_{\dot \Phi}(A\mathscr{W} ) \, G- \big[\bar{\phi}+ ( \beta - 1) \Phi\big] \dot{G}\mathscr{W}^{2},
\end{align*}
where Euler's relation $\dot{G}\mathscr{W}= \alpha G$ is
used in the last line.
\end{proof}

An immediate application of the theorem above is to obtain the
evolution equations for $H$, and $F$.
\begin{proposition}\label{evolution equations form mvpcf}
On any
solution $M_{t}$ of \eqref{mvpcf}-\eqref{barPhi} with $\Phi$ given by
\eqref{def_Phi} the following hold:
\begin{align}
\notag\partial_{t}H=& \Delta_{\dot{\Phi}} H + \tr \big[\ddot{\Phi}(\nabla_{_{\!
\ds \cdot}} \mathscr{W}, \nabla^{_{\! \ds \cdot}} \mathscr{W})\big]
+\tr_{\dot \Phi}(A\mathscr{W} )\, H - \big(\bar{\phi} + (\beta -1) \Phi\big) |A|^2,\\
 \label{evturHn}\partial_{t}{H}^{n}=& \Delta_{\dot{\Phi}} {H}^{n}  -n(n-1){H}^{n-2}|\nabla{H}|_{\dot{\Phi}}^{2}
   + n{H}^{n-1}\tr\big[\ddot{\Phi}(\nabla_{_{\! \ds \cdot}} \mathscr{W},
\nabla^{_{\! \ds \cdot}} \mathscr{W})\big]
\\& + n\big(\Phi-\bar{\phi}\big){H}^{n-1} |A|^2\notag\\
\label{evF}
\partial_{t}F=& \Delta_{\dot{\Phi}} F -\dot{\Phi}\ddot{F}(\nabla_{_{\! \ds \cdot}} \mathscr{W} ,\nabla^{_{\!
\ds \cdot}} \mathscr{W})
  + \dot{F} \ddot{\Phi}(\nabla_{_{\! \ds \cdot}} \mathscr{W} ,\nabla^{_{\! \ds \cdot}} \mathscr{W})\\
 \qquad \qquad\qquad &+\tr_{\dot \Phi}(A\mathscr{W} )\, F - \big(\bar{\phi} + (\beta -1) \Phi\big) \tr_{\dot \Phi}( A\mathscr{W}),\notag\\
 \label{evturFn}\partial_{t}{F}^{n}=& \Delta_{\dot{\Phi}} {F}^{n}
  -n(n-1){F}^{n-2}|\nabla{F}|_{\dot{\Phi}}^{2}
   -n{F}^{n-1}\dot{\Phi}\ddot{F}(\nabla_{_{\! \ds \cdot}} \mathscr{W} ,\nabla^{_{\!
\ds \cdot}} \mathscr{W})\\&
   + n{F}^{n-1}\tr_{\dot{F}}\big[\ddot{\Phi}(\nabla_{_{\! \ds \cdot}} \mathscr{W},
\nabla^{_{\! \ds \cdot}} \mathscr{W})\big]+n F^{n}\tr_{\dot \Phi}( A\mathscr{W})
- \big(\bar{\phi} + (\beta -1) \Phi\big) \tr_{\dot \Phi}( A\mathscr{W}),\notag\\
\label{evturK for the powerF}\partial_{t}K=&\,\Delta_{\dot{\Phi}} K
-\dot{\Phi}\ddot{K}(\nabla_{_{\! \ds \cdot}}
\mathscr{W} ,\nabla^{_{\! \ds \cdot}} \mathscr{W})
  + \dot{K} \ddot{\Phi}(\nabla_{_{\! \ds \cdot}} \mathscr{W} ,\nabla^{_{\! \ds \cdot}} \mathscr{W})
\\&
+ \big[(1- \beta) \Phi-\bar{\phi} \big]
\dot{K}\mathscr{W}^{2}
 +\tr_{\dot \Phi}( A\mathscr{W} )
\dot{K}\mathscr{W},\notag\\
  \label{ev Phi}
\partial_{t}\Phi = &\Delta_{\dot \Phi} \Phi + (\Phi - \bar{\phi}) \,\tr_{\dot
\Phi} (A \mathscr{W}).
\end{align}
\end{proposition}

Furthermore, \eqref{evturK for the powerF} can be rewritten as
\begin {lemma}
On any solution $M_{t}$ of \eqref{mvpcf}-\eqref{barPhi} with $\Phi$ given by
\eqref{def_Phi}  the following holds:
\begin{align}
\label{evturK2 for the powerF}
\partial_{t}K=&\,\Delta_{\dot{\Phi}} K
-\frac{(n-1)}{n}\frac{\left|\nabla K\right|_{\dot
\Phi}^{2}}{K}
+\frac{K}{H^{2}}\left|H\nabla\mathscr{W} -\mathscr{W}\nabla H\right|^{2}_{\dot \Phi, b}\\
&-\frac{H^{2n}}{nK}\left|\nabla(KH^{-n})\right|_{\dot
\Phi}^2 +K\,\tr_{ b}\left(
\ddot{\Phi}(\nabla\mathscr{W} ,\nabla
\mathscr{W})\right) \notag
\\& + \big[(1- \beta) \Phi-\bar{\phi} \big]
KH
+n\tr_{\dot \Phi}( A\mathscr{W} ) K,\notag\\
\label{evturK3 for the powerF} =&\,\Delta_{\dot{\Phi}} K
-\frac{(n-1)}{n}\frac{\left|\nabla K\right|_{\dot
\Phi}^{2}}{K}
+\frac{K}{F^{2}}\left|F\nabla\mathscr{W} -\mathscr{W}\nabla F\right|^{2}_{\dot \Phi, b}\\
&-\frac{F^{2n}}{nK}\left|\nabla(KF^{-n})\right|_{\dot
\Phi}^2 +K\,\tr_{ b}\left(
\ddot{\Phi}(\nabla\mathscr{W} ,\nabla
\mathscr{W})\right) \notag
\\& + \big[(1- \beta) \Phi-\bar{\phi} \big]
KH
+n\tr_{\dot \Phi}( A\mathscr{W} ) K,\notag
\end{align}
where $ b:= \mathscr{W}^{-1}$.
\end {lemma}
\begin{proof}
Note that
\begin{equation}
\dot{K}=K b,
\end{equation}
this implies
\begin{equation}\label{evtK for the powerF}
\dot{K}\mathscr{W}^{2}=KH,
\end{equation}
and
\begin{equation}\label{evtK1 for the powerF}
\dot{K} \ddot{\Phi}(\nabla_{_{\! \ds \cdot}}
\mathscr{W} ,\nabla^{_{\! \ds \cdot}} \mathscr{W})
=K b\ddot{\Phi}(\nabla_{_{\! \ds \cdot}}
\mathscr{W} ,\nabla^{_{\! \ds \cdot}}
\mathscr{W})=K\,\tr_{ b}\left(
\ddot{\Phi}(\nabla\mathscr{W} ,\nabla
\mathscr{W})\right).
\end{equation}
A direct calculation as for example in Lemma $3.2$ of \cite{Cho85}
gives
\begin{equation}\label{evtK2 for the powerF}
-\dot{\Phi}\ddot{K}(\nabla_{_{\! \ds \cdot}}
\mathscr{W} ,\nabla^{_{\! \ds \cdot}} \mathscr{W})
=-\frac{\left|\nabla K\right|_{\dot
\Phi}^{2}}{K}-K\,\tr_{\dot{\Phi}}\left( \nabla b
\nabla \mathscr{W}\right)
\end{equation}
and
\begin{equation}\label{evtK3 for the powerF}
-K\,\tr_{\dot{\Phi}}\left( \nabla b \nabla
\mathscr{W}\right)=\frac{K}{H^{2}}\left|H\nabla\mathscr{W}
-\mathscr{W}\nabla H\right|^{2}_{\dot \Phi,b}
+\frac{\left|\nabla K\right|_{\dot
\Phi}^{2}}{nK}-\frac{H^{2n}}{nK}\left|\nabla(KH^{-n})\right|_{\dot
\Phi}^2.
\end{equation}
Therefore, identities \eqref{evtK for the powerF}, \eqref{evtK1 for the powerF}, \eqref{evtK2 for the powerF} and
\eqref{evtK3 for the powerF} together apply to \eqref{evturK for the powerF} to give
\eqref{evturK2 for the powerF}.

Also, \begin{equation}\label{evtK4 for the powerF}
-K\,\tr_{\dot{\Phi}}\left( \nabla b \nabla
\mathscr{W}\right)=\frac{K}{F^{2}}\left|F\nabla\mathscr{W}
-\mathscr{W}\nabla F\right|^{2}_{\dot \Phi,b}
+\frac{\left|\nabla K\right|_{\dot
\Phi}^{2}}{nK}-\frac{F^{2n}}{nK}\left|\nabla(KF^{-n})\right|_{\dot
\Phi}^2.
\end{equation}
Similarly, applying identities \eqref{evtK for the powerF}, \eqref{evtK1 for the powerF}, \eqref{evtK2 for the powerF} and
\eqref{evtK4 for the powerF} together to \eqref{evturK for the powerF} gives
\eqref{evturK3 for the powerF}.

\end{proof}

\section{\bf Preserving pinching}\label{Preserving pinching for powerF flow}

To control the pinching of the principal curvatures along the flow
\eqref{mvpcf} of Euclidean spaces,  Schulze, in \cite{Sch06},
following an idea of Tso \cite{Tso}, looked at a test function
$K/H^{n}$, which was also considered in \cite{CS10}. Furthmore two analogous
quantities $\tilde{K}/\tilde{H}^{n}$ and $K/ F^{n}$
were taken into consideration in \cite{GLW16,GuoLW16} and \cite{McC05} respectively.
In this section, we use test functions $Q_{1}=K/H^{n}$  in the case of concave $F$ and $Q_{2}=K/F^{n}$  in the case of convex $F$ respectively.
The rest of this section consists of showing the inequality
$ Q_{i}\geq C>0$ remain under the evolution for $i=1,2$.

We begin by deriving the following evolution equations for the quantities $ Q_{i}$ for $i=1,2$.
\begin {lemma}
For the ambient space $N^{n+1}=\mathbb{R}^{n+1}$, on any
solution $M_{t}$ of \eqref{mvpcf} the following holds:
\begin{align}\label{evtildeq}
\partial_{t} Q_{1}=&\Delta_{\dot \Phi}  Q_{1}
+\frac{(n+1)}{n H^{n}}\left\langle\nabla
 Q_{1},\nabla H^{n}\right\rangle_{\dot \Phi}
-\frac{(n-1)}{n K}\left\langle\nabla  Q_{1},\nabla K\right\rangle_{\dot \Phi}
-\frac{ H^{n}}{n K}\left|\nabla Q_{1}\right|_{\dot
\Phi}^2
\notag\\
&+\frac{ Q_{1}}{ H^{2}}\left| H\nabla\mathscr{W}
-\mathscr{W}\nabla H\right|^{2}_{\dot \Phi, b}
+ Q_{1}\,\tr_{ b-\frac{n}{ H}Id}\left(
\ddot{\Phi}(\nabla\mathscr{W} ,\nabla \mathscr{W})\right)\\
&+\big[( \beta-1) \Phi+\bar{\phi} \big]
\frac{ Q_{1}}{ H}\left(n\bigl| A \bigl|^{2}- H^{2}\right).\notag
\end{align}
and
\begin{align}\label{evtildeq2}
\partial_{t} Q_{2}=&\Delta_{\dot \Phi}  Q_{2}
+\frac{(n+1)}{n F^{n}}\left\langle\nabla
 Q_{2},\nabla F^{n}\right\rangle_{\dot \Phi}
-\frac{(n-1)}{n K}\left\langle\nabla  Q_{2},\nabla K\right\rangle_{\dot \Phi}
-\frac{ F^{n}}{n K}\left|\nabla Q_{2}\right|_{\dot
\Phi}^2
\notag\\
&+\frac{ Q_{2}}{ F^{2}}\left| F\nabla\mathscr{W}
-\mathscr{W}\nabla F\right|^{2}_{\dot \Phi, b}
+ Q_{2}\,\tr_{ b-\frac{n}{ F}\dot F}\left(
\ddot{\Phi}(\nabla\mathscr{W} ,\nabla \mathscr{W})\right)\\
& +n\frac{ Q_{2}}{F}\ddot{F}(\nabla \mathscr{W} ,\nabla\mathscr{W})
+\big[( \beta-1) \Phi+\bar{\phi} \big]
\frac{ Q_{2}}{ F}\left(n\tr_{\dot F}( A\mathscr{W} ) - H F\right).\notag
\end{align}
\end {lemma}
\begin{proof}
We prove the evolution equation \eqref{evtildeq2} for $Q_{2}$;
the proof of the evolution equation \eqref{evtildeq} for $Q_{1}$ is similar by using \eqref{evturHn} and \eqref{evturK2 for the powerF}.
 By  \eqref{evturFn} and \eqref{evturK2 for the powerF}
\begin{align}
\partial_{t} Q_{2}&=\frac{1}{ F^{n}}\partial_{t} K
-\frac{1}{ F^{2n}}\partial_{t} F^{n}\notag\\
=&\frac{ \Delta_{\dot{\Phi}}  K}{ F^{n}}
-\frac{ K}{ F^{2n}}\Delta_{\dot{\Phi}}  F^{n}
 -\frac{(n-1)}{n}\frac{\left|\nabla K\right|_{\dot{\Phi}}^2}{ K F^{n}}
-\frac{ Q_{2}}{n}\left|\nabla Q_{2}\right|_{\dot{\Phi}}^2
+n(n-1)\frac{ Q_{2}}{ F^{2}}\left|\nabla F\right|_{\dot{\Phi}}^2\notag
\\\label{evtildeQ}
&+\frac{ Q_{2}}{ F^{2}}\left| F\nabla\mathscr{W}
-\mathscr{W}\nabla F\right|^{2}_{\dot \Phi, b}
+ Q_{2}\,\tr_{ b-\frac{n}{ F}\dot F}\left(
\ddot{\Phi}(\nabla\mathscr{W} ,\nabla \mathscr{W})\right)\\
&+n\frac{ Q_{2}}{F}\ddot{F}(\nabla \mathscr{W} ,\nabla\mathscr{W})+\big[( \beta-1) \Phi+\bar{\phi} \big]
\frac{ Q_{2}}{ F}\left(n\tr_{\dot F}( A\mathscr{W} ) - H F\right).\notag
\end{align}
Furthermore, the first derivative and second derivative term in
\eqref{evtildeQ} can be computed as follows, the equality
\[
\nabla\left(\frac{ K}{ F^{n}}\right)
=\frac{\nabla K}{ F^{n}}
-\frac{ K}{ F^{2n}}\nabla F^{n}
\]
implies
\begin{align}\label{secdriveQ}
\Delta_{\dot \Phi}\left(\frac{ K}{ F^{n}}\right)
&=\frac{\Delta_{\dot \Phi}  K}{ F^{n}}
-2\frac{\left\langle\nabla  F^{n},\nabla\tilde
K\right\rangle_{\dot \Phi}}{ F^{2n}}
+2\frac{ K}{ F^{3n}}
\big|\nabla F^{n}\big|_{\dot \Phi}^{2}
-\frac{ K}{ F^{2n}}\Delta_{\dot \Phi}  F^{n},
\end{align}
\begin{equation}\label{gradientQ1}
\begin{split}
\left\langle\nabla\left(\frac{ K}{ F^{n}}\right),\nabla F^{n}
\right\rangle_{\dot \Phi}=\frac{\left\langle\nabla
 F^{n},\nabla\tilde K\right\rangle_{\dot \Phi}}{ F^{n}}
-\frac{ K}{ F^{2n}}\big|\nabla F^{n}\big|_{\dot
\Phi}^{2},
\end{split}
\end{equation}
and
\begin{equation}\label{gradientQ2}
\begin{split}
\left\langle\nabla\left(\frac{ K}{ F^{n}}\right),\nabla K\right\rangle_{\dot
\Phi} =\frac{\big|\nabla K\big|_{\dot \Phi} ^{2}}{ F^{n}}
-\frac{ K}{ F^{2n}}\left\langle\nabla
 F^{n},\nabla\tilde K\right\rangle_{\dot \Phi}.
\end{split}
\end{equation}
From \eqref{secdriveQ}, \eqref{gradientQ1} and \eqref{gradientQ2},
it follows
\begin{equation}\label{derivateterm}
\begin{split}
\frac{ \Delta_{\dot \Phi} K}{ F^{n}}&
-\frac{ K}{ F^{2n}}\Delta_{\dot \Phi} F^{n}
 -\frac{(n-1)}{n}\frac{\left|\nabla K\right|_{\dot
\Phi}^2}{ K F^{n}}\\
&=\Delta_{\dot \Phi} \left(\frac{ K}{ F^{n}}\right)
+\frac{(n+1)}{n F^{n}}\left\langle\nabla
\left(\frac{ K}{ F^{n}}\right),\nabla F^{n}\right\rangle_{\dot
\Phi}\\
&\quad-\frac{(n-1)}{n K}\left\langle\nabla
\left(\frac{ K}{ F^{n}}\right),\nabla\tilde
K\right\rangle_{\dot \Phi}
-n(n-1)\frac{ K}{ F^{n+2}}\big|\nabla F\big|_{\dot
\Phi}^{2}.
\end{split}
\end{equation}
Thus, applying equation \eqref{derivateterm} to \eqref{evtildeQ} gives \eqref{evtildeq2}.
\end{proof}
In order to apply the maximum principle to \eqref{evtildeq} and  \eqref{evtildeq2}, and show
that for $i=1,2$,  $\min_{p\in M_{t}} Q_{i}(p,t)$ are non-decreasing in time
some preliminary inequalities are needed in the sequel. The
following elementary property is a consequence of (\cite{CS10},
Lemma 4.2) (see also \cite{Cho85} and \cite{Sch06}).

\begin {lemma}\label{pinching imply convex lemma}
Given $\varepsilon \in (0,1/n)$,  for any convex (concave) $F$  fulfilling the conditions \ref{Fconds},
there exists a constant $C_{1}(\varepsilon, n)$ ($C_{2}(\varepsilon, n)$), such that,
for any $\lambda =(\lambda_{1},\dots, \lambda_{n}) \in \mathbb{R}^n$
with $0<\lambda_{1}\leq \cdots \leq \lambda_{n}$,
\[
 K(\lambda) > C_{1} F^n(\lambda)~\big( K(\lambda) > C_{2} H^n(\lambda)\big),
\]
then, we have
\[
\lambda_{1}>\varepsilon n F(\lambda)~\big( \lambda_{1}>\varepsilon H(\lambda)\big).
\]
\end {lemma}

\begin{proof}
 Given any $\varepsilon \in (0,1/n)$, we define
$$
\Gamma_\varepsilon= \{ \lambda=(\lambda_1,\dots,\lambda_n) ~:~  0 \leq   \lambda_1 \leq \varepsilon n F(\lambda) \},
$$
$$
{\Xi}_\varepsilon=\{ \lambda\in {\Gamma}_\varepsilon  ~:~ |\lambda|=1 \}.
$$
We have $F(\lambda)>0$ on any nonzero element of ${\Gamma}_\varepsilon$;
 hence, the quotient $K/F^n$ is defined everywhere on ${\Xi}_\varepsilon$.
 Let us call $M_\varepsilon$ the maximum of $K/F^n$ on ${\Xi}_\varepsilon$,
which exists because ${\Xi}_\varepsilon$ is compact.
By homogeneity, the inequality $K \leq M_\varepsilon F^n$ is also satisfied by the elements of ${\Gamma}_\varepsilon$.
Therefore, if $\lambda=(\lambda_1,\dots,\lambda_n)$ with with $0<\lambda_{1}\leq \cdots \leq \lambda_{n}$ is such that $K > M_\varepsilon F^n$,
then $\lambda$ does not belong to ${\Gamma}_\varepsilon$.  The lemma follows by choosing $C=M_\varepsilon$.
\end{proof}

The following estimate which is a stronger version of Lemma\,2.3
(ii) in \cite{Hui84} can be viewed as a generalization of
Cabezas-Rivas and Miquel in \cite{CS10}.
\begin {lemma}\label{a key estimate}
If $ F$ is convex and positive, and the inequality
$\mathscr{W}>\varepsilon n F Id$ is valid with some
$\varepsilon>0$ at a point on a hypersurface immersed in $\mathbb{R}^{n+1}$,
then $\varepsilon\leq 1/n$ and
\[
\left| F\nabla\mathscr{W}
-\mathscr{W}\nabla F\right|^{2} \geq
\frac{n-1}{2}\varepsilon^{2} F^{2}\left|\nabla\mathscr{W}\right|^{2}.
\]
\end {lemma}
\begin{proof}
Proceeding as  in \cite [Lemma\,4.1] {CS10}, we first observe that the assumption implies that $\lambda_1 \geq \varepsilon n F>0$.
It follows form this, Conditions \ref{Fconds} \ref{item3} and \ref{item4} that
$$
F(\lambda_1,\cdots, \lambda_n)\geq \lambda_1 F(1,\cdots, 1)\geq  \lambda_1 \geq n \varepsilon F(\lambda_1,\cdots, \lambda_n)>0,
$$
which implies that $\varepsilon \leq 1/n$.
Furthermore applying Lemma \ref{FHInequalities} with the convexity of $F$ to $
\lambda_1 \geq n \varepsilon F$ implies that $\lambda_1 \geq \varepsilon H$. Thus we may deduce that
\begin{align}
h^j_i h^l_j \nabla^i F \nabla_l F & \leq \lambda_n^2 |\nabla F|^2 \leq (|A|^2-(n-1) \lambda_1^2) |\nabla F|^2
\nonumber \\
& \leq (|A|^2-(n-1) \varepsilon^2 H^2 ) |\nabla F|^2  \leq  (1-(n-1) \varepsilon^2) |A|^2 |\nabla F|^2. \label{laux1}
\end{align}
Now we can write
\begin{equation}
| F \nabla \mathscr{W}  - \mathscr{W} \nabla  F|^2 = |\nabla \mathscr{W}|^2  F^2 + |\nabla  F|^2 |A|^2 - 2  F \langle\mathscr{W} \nabla  F, \nabla \mathscr{W}\rangle \label{laux2}
\end{equation}
whose last term, in local coordinates, takes the following form
$$-2  F \nabla_{i}  h_{jl} \nabla^{i}  F  h^{jl} = -  F \nabla_{i}  h_{jl}\big(\nabla^{i}  F  h^{jl} + \nabla^{j}  F  h^{il}\big),$$
by the Codazzi equations. Using this and  the inequality
 $\langle U, V\rangle \leq \frac{2 - \varepsilon'}{2} |U|^2 + \frac1{2(2 - \varepsilon')} |V|^2$, with $U= F \nabla_i  h_{jl}$, $V=
\nabla^i  F  h^{jl} + \nabla^j  F  h^{il}$ and $\varepsilon' = (n-1)\varepsilon^2$, one estimate
\begin{align*}
2  F \langle\mathscr{W} \nabla  F, \nabla \mathscr{W}\rangle & \leq \frac{2 - \varepsilon'}{2}  F^2 |\nabla \mathscr{W}|^2 + \frac1{2(2 - \varepsilon')} |\nabla^i  F  h^{jl} + \nabla^j  F  h^{il}|^2
\\& = \big(1 - \frac{\varepsilon'}{2}\big)  F^2 |\nabla \mathscr{W}|^2 + \frac1{2 - \varepsilon'}\big(|\nabla  F|^2 |A|^2 +  \nabla_i  F  h_{jl} \nabla^j  F  h^{il}\big)\\
&\! \!\os{\eqref{laux1}}{\leq} \big(1 - \frac{\varepsilon'}{2}\big)  F^2 |\nabla \mathscr{W}|^2 + |\nabla  F|^2 |A|^2,
\end{align*}
and the conclusion follows from \eqref{laux2}.

\end{proof}

The estimate above is also valid in the case of concave $F$, see also \cite [Lemma\,4.1] {CS10}.
\begin {lemma}\label{a key estimate2}
If $ F$ is concave and positive, and the inequality
$\mathscr{W}>\varepsilon HId$ is valid with some
$\varepsilon
>0$ at a point on a hypersurface immersed in $\mathbb{R}^{n+1}$,
then $\varepsilon\leq 1/n$ and
\[
\left| H\nabla\mathscr{W}
-\mathscr{W}\nabla H\right|^{2} \geq
\frac{n-1}{2}\varepsilon^{2} H^{2}\left|\nabla\mathscr{W}\right|^{2}.
\]
\end {lemma}


Also as in \cite{CS10}, the preceding three lemmas allow us to prove
the pinching estimate for our flow, which is one of the key steps in
the proof of our main result.
\begin{theorem}
 \label{pinching} There exists a constant $C^{*}=C(n,\beta)$ with the following property: if $X:M \times (0,T)\,
\rightarrow \mathbb{R}^{n+1}$,
with $t \in (0,T)$, is a smooth solution of \eqref{mvpcf}--\eqref{barPhi} with $\Phi(F)$ given by \eqref{def_Phi} for some $\beta\geq1$,
where $F$ satisfies Conditions \ref{Fconds} \ref{item1}, \ref{item2}, \ref{item3},
and $f(1, \cdots, 1)=1$, such that
\begin{itemize}
\item the initial immersion $X_0$ satisfies \eqref{convex pinched} for $F$ convex (\eqref{concave pinched} for $F$ concave) with the constant $C^{*}$,
\item the solution $M_t=X(M,t)$ satisfies $F>0$ ($F>0$ ) for all times $t \in (0,T)$,
\end{itemize}
then the minimum of $ Q_2$  ($ Q_1$ ) on $M_t$ is
nondecreasing in time.
\end{theorem}
\begin{proof}
We prove the convex $F$ case; the concave $F$ case is similar.
The assumption $ F>0$ on initial hypersurface ensures that
the quotient $ Q_{2}$ is well-defined for $t \in (0,T)$. For
proof of the theorem, it is sufficient to prove that the minimum of
$ Q_{2}$ (denote by $\mathcal {Q}$) is nondecreasing in time.
First, applying \eqref{convex pinched} with the assumption $ F>0$
to Lemma \ref{pinching imply convex lemma} implies that $\lambda_{1}>0$ on $M_{t}$ for $t=0$,
then this implies that $\lambda_{1}>0$ on
$M_{t}$ for $t \in (0,T)$. In fact, suppose to the contrary that
there exists a first time $t_{0}>0$ at which $\lambda_{1}=0$
at some point, then $\mathcal {Q}(t_{0})=0$. On the other
hand, since the theorem holds in the convex case,
$\mathcal {Q}(t)$ is nondecreasing in $(0,t_{0})$, so it
cannot decrease from $C^{*}$ to zero which gives a contradiction.
Now applying the maximum principle to equation \eqref{evtildeq2} for
$\mathcal {Q}$ gives

\begin{align}
\partial_{t} Q_{2}\geq
&\frac{ Q_{2}}{ F^{2}}\left| F\nabla\mathscr{W}
-\mathscr{W}\nabla F\right|^{2}_{\dot \Phi, b}
+ Q_{2}\,\tr_{ b-\frac{n}{ F}\dot F}\left(
\ddot{\Phi}(\nabla\mathscr{W} ,\nabla \mathscr{W})\right)\notag\\
&+\big[( \beta-1) \Phi+\bar{\phi} \big]
\frac{ Q_{2}}{ F}\left(n\tr_{\dot F}( A\mathscr{W} ) - H F\right).\notag\\
\label{evtildeq estimate}&\geq
 Q_{2}\bigg\{\frac{1}{ F^{2}}\left| F\nabla\mathscr{W}
-\mathscr{W}\nabla F\right|^{2}_{\dot \Phi, b}
-\left| b-\frac{n}{ F}\dot F\right|\,
\left|\ddot{\Phi}(\nabla\mathscr{W} ,\nabla \mathscr{W})\right|\\
&+\big[( \beta-1) \Phi+\bar{\phi} \big]
\frac{ Q_{2}}{ F}\left(n\tr_{\dot F}( A\mathscr{W} ) - H F\right).\notag
\end{align}
The various terms appearing here can be estimated as follows, as in
\cite [Theorem 4.3] {CS10}.
Combining the convexity of $M_{t}$ and Lemma \ref{FHInequalities2},
we estimate that
$n\tr_{\dot F}( A\mathscr{W} ) - H F \geq 0$.
Thus the third term  on right hand side  in inequality \eqref{evtildeq estimate}
can be dropped with the strictly convexity on $M_{t}$ and the assumption $\beta\geq1$.
It remains to estimate the first two terms  on right hand side in the inequality
\eqref{evtildeq estimate}, now proceeding exactly as in \cite{CS10},
\cite{Cho85} and \cite{Sch06}, choose orthonormal frame which
diagonalizes $\mathscr{W}$ so that
\begin{align}\label{mintildeQ estimate}
\left| F\nabla\mathscr{W}
-\mathscr{W}\nabla F\right|^{2}_{\dot \Phi, b}
&=\sum_{i,m,n}{\dot \Phi}^{i}\frac{1}{\lambda_{m}}\frac{1}{\lambda_{n}}
\left( F\nabla_{i}{ F}_{m}^{n}
-{ F}_{m}^{n}\nabla_{i} F\right)^{2}\\
&\geq\frac{1}{ H^{2}}\sum_{i,m,n}{\dot \Phi}^{i}
\left( F\nabla_{i}{ F}_{m}^{n}
-{ F}_{m}^{n}\nabla_{i} F\right)^{2}\notag
\end{align}
where $\lambda_{m}\leq  H$ was used in the last
inequality by strictly convexity of $M_{t}$, i.e.,
$\lambda_{m}>0$ for any $m$. Now the property that each
$\dot{\Phi}^i$ is positive in the interior of the positive cone can be
used. More precisely, for any $\varepsilon \in (0,1/n]$, we set
$$
{\Xi}_\varepsilon:=\{ \lambda
=(\lambda_1,\dots,\lambda_n) \in \mathbb{R}^n ~:~
 \lambda_1 \geq
\varepsilon n F >0 \, \},
$$
and
$$
W_1(\varepsilon)=\min \{ {\dot \Phi}^i (\lambda) ~:~  1 \leq i
\leq n, \ {\lambda} \in {\Xi}_\varepsilon, |
{\lambda}|=1 \}.
$$
By homogeneity of ${\dot \Phi}^i$ with degree $\beta-1$
and Conditions \ref{Fconds} \ref{item3}, exactly as in  the formula at the top of
p.453 of \cite {CS10}, the following inequality holds:
$$
{\dot \Phi}^i( \lambda) \geq W_1(\varepsilon)
|\lambda|^{\beta-1}, \qquad {\lambda} \in
{\Xi}_\varepsilon,
$$
where $W_1(\varepsilon)$ is an increasing positive function of
$\varepsilon$. This estimation, convexity of a hypersurface and
Lemma \ref{a key estimate} together imply that the inequality
\eqref{mintildeQ estimate} can be estimated as follows:
\begin{equation}\label{mintildeQ estimate2}
\begin{split}
\left| F\nabla\mathscr{W}
-\mathscr{W}\nabla F\right|^{2}_{\dot \Phi, b}
\geq\frac{n-1}{2} W_1(\varepsilon)\varepsilon^{2}
 F^{2}|\mathscr{W}|^{\beta-3}\left|\nabla\mathscr{W}\right|^{2},
\end{split}
\end{equation}
for some $\varepsilon \in (0, 1/n)$.

The term $\left| \ddot{\Phi}(\nabla\mathscr{W} ,\nabla
\mathscr{W})\right|$ is smooth as long as
$\lambda_{i}>0$ for any $i$, homogeneous of degree
$\beta-2$ in $\lambda_{i}$ and quadratic in
$\nabla\mathscr{W}$. Thus the following estimation of the
term $\left| \ddot{\Phi}(\nabla\mathscr{W} ,\nabla
\mathscr{W})\right|$ can be derived as in \cite[inequality
(4.7)]{CS10}: For any $\varepsilon \in (0, 1/n)$, there exists a
constant $W_2(\varepsilon)$ such that, at any point where
$\mathscr{W}\geq \varepsilon n F Id$,
\begin{equation}\label{secder F}
\left| \ddot{\Phi}(\nabla\mathscr{W} ,\nabla
\mathscr{W})\right| \leq W_2(\varepsilon)
|\mathscr{W}|^{\beta-2}\left|\nabla\mathscr{W}\right|^{2},
\end{equation}
where $W_2(\varepsilon)$ is decreasing in $\varepsilon$.

A next step is to show that
$\left| b-\frac{n}{ F}\dot F\right|$ is small if the
principal curvatures are pinched enough.
Since Lemma \ref{relation of dotF and dotf} with $F$ convex implies that the inequalities
$\dot f_{1}\leq \cdots \leq \dot f_{n} $.
Furthermore, by Conditions \ref{Fconds} \ref{item3}, it is clear that
\[
\left| b-\frac{n}{ F}\dot F\right| \leq
\sqrt{n}\max\left\{\left(\frac{1}{\lambda_{1}}-\frac{n \dot f_{1}}{f}\right),\left(\frac{n \dot f_{n}}{f}-\frac{1}{\lambda_{n}}\right) \right\}.
\]
Since for some $\varepsilon \in (0, 1/n)$
\begin{equation}\label{lower estimate on min.tilde curvature}
\lambda_{1}\geq \varepsilon n f,
\end{equation}
then
\begin{equation}\label{lower estimate on min.tilde curvature1}
\frac{1}{\lambda_{1}}-\frac{n \dot f_{1}}{f}\leq
\frac{1-\varepsilon n^{2}\dot f_{1}}{\varepsilon n f}.
\end{equation}
Furthermore, by homogeneity of $\dot f_{1}$ with degree $0$
and Conditions \ref{Fconds} \ref{item3},
there exists a constant $W_3(\varepsilon)$ such that,
at any point where \eqref{lower estimate on min.tilde curvature} holds,
$$
\dot f_{1}( \lambda) \geq W_3(\varepsilon), \qquad {\lambda} \in
{\Xi}_\varepsilon,
$$
where $W_3(\varepsilon)$ is an increasing positive function of $\varepsilon$.
Thus, \eqref{lower estimate on min.tilde curvature1} can be estimated as follows:
\begin{equation}\label{lower estimate on min.tilde curvature2}
\frac{1}{\lambda_{1}}-\frac{n \dot f_{1}}{f}\leq
\frac{1-\varepsilon n^{2}W_3(\varepsilon)}{\varepsilon n f}.
\end{equation}
On other hand, recalling the derivation of \eqref{laux1},
the convexity of $M_{t}$ implies that
\begin{equation*}
\lambda_{n}\leq \left(1-(n-1)\varepsilon\right) H
\end{equation*}
Thus Lemma \ref{FHInequalities} with $F$ convex gives as follows:
\begin{equation*}
\lambda_{n}\leq \left(1-(n-1)\varepsilon\right)n f
\end{equation*}
which implies that
\begin{equation*}
\frac{n \dot f_{n}}{f}-\frac{1}{\lambda_{n}}\leq
\frac{(n-1)\left(1-n\varepsilon\right)}{ n \left(1-(n-1)\varepsilon\right) f},
\end{equation*}
where we use the inequality $\dot f_{n} \leq 1$ by Lemma \ref{FHInequalities} with $F$ convex.
This combines with estimate \eqref{lower estimate on min.tilde
curvature2} to give
\begin{equation}\label{eq:2.27}
\left| b-\frac{n}{ F}\dot F\right| \leq
\frac{\mathscr{N}(\varepsilon)}{ F},
\end{equation}
where
\begin{equation}
\mathscr{N}(\varepsilon):=
\frac{1}{\sqrt{n}} \max\left\{\frac{(1-\varepsilon n^{2}W_3(\varepsilon))}{\varepsilon },
\frac{(n-1)\left(1-n\varepsilon\right)}{\left(1-(n-1)\varepsilon\right)} \right\}.
\end{equation}
 Thus, recalling the facts that the convexity of  $F$ implies the inequalities $ F\geq \frac{H}{n}$,
and $\big|H\big|\geq \big|\mathscr{W}\big|$, estimations
\eqref{mintildeQ estimate}, \eqref{mintildeQ estimate2},
\eqref{secder F} and \eqref{eq:2.27} together give:
\begin{align*}
\frac{1}{ F^{2}}\left| F\nabla\mathscr{W}
-\mathscr{W}\nabla F\right|^{2}_{\dot \Phi, b}
&-\left| b-\frac{n}{ F}\dot F\right|\,
\left|\ddot{\Phi}(\nabla\mathscr{W} ,\nabla \mathscr{W})\right|\notag
\\&\geq
|\mathscr{W}|^{\beta-3}\left|\nabla\mathscr{W}\right|^{2}
\left(\frac{(n-1)}{2}W_1(\varepsilon)\varepsilon^{2}-n W_2(\varepsilon)\mathscr{N}(\varepsilon)\right).
 \end{align*}
To achieve our purpose by application of the maximum principle,
it is necessary that
$\mathscr{F}(\varepsilon):=|\mathscr{W}|^{\beta-3}\left|\nabla\mathscr{W}\right|^{2}
\left(\frac{(n-1)}{2}W_1(\varepsilon)\varepsilon^{2}-n W_2(\varepsilon)\mathscr{N}(\varepsilon)\right)$ is non-negative on
$M_{t}$. In fact, $\mathscr{N}(\varepsilon)$ is a strictly
decreasing function of $\varepsilon$; in addition,
$\mathscr{N}(\varepsilon)$ is arbitrarily large as $\varepsilon$
goes to zero and tends to zero as $\varepsilon$ goes to $1/n$ by definition, $W_1(\varepsilon)$ is increasing and
$W_2(\varepsilon)$ is decreasing. Therefore,
$\mathscr{F}(\varepsilon)$ is a strictly increasing function
of $\varepsilon$, it is negative as $\varepsilon$ goes to zero and
positive as $\varepsilon$ goes to $1/n$. So there exists a unique
value $\varepsilon_{0}\in (0, 1/n)$ such that
\begin{equation} \label{varepsilon0}
\mathscr{F}(\varepsilon_{0})=0.
 \end{equation}
By Lemma \ref{pinching imply convex lemma} there exists a constant
$C^{*}\in  (0, 1/n^n)$ satisfies $  Q_{2}(\lambda) >
C^{*} $ such that $ \lambda_{1}
>\varepsilon F(\lambda)
$ with a $\varepsilon_{0}\in (0, 1/n)$ given by \eqref{varepsilon0}.
Thus, if $ Q_{2}> C^{*}\geq 0$ everywhere on the initial
hypersurface, applying the maximum principle for $ Q_{2}$ implies
that $\partial_{t}\mathcal {Q}\geq 0$, i.e.,
$\mathcal {Q}$ is nondecreasing in time. This guarantees
that $ Q_{2}> C^{*}$ is preserved
on any solution $M_{t}$ of \eqref{mvpcf}-\eqref{barPhi} with $\Phi$ given by \eqref{def_Phi}
in $\mathbb{R}^{n+1}$.
\end{proof}

Theorem \ref{pinching} asserts that for convex $F$ inequality $ Q_{2}> C^{*}$ ($Q_{1}> C^{*}$ for concave $F$)
holds for all $t \in [0,T)$, furthermore the definition of $C^{*}$
together with Lemma \ref{pinching imply convex lemma} shows that
\begin{equation}\label{tildle lambda pinching}
\lambda_{i}\geq \varepsilon_{0}  n F\geq \varepsilon_{0}  H\, (\lambda_{i}\geq \varepsilon_{0}  H\geq \varepsilon_{0}  n F)\quad \text{on}\
M\times [0,T) \quad \text{for each}\ i,
\end{equation}
where $\varepsilon_{0}$ is given by \eqref{varepsilon0}.
In particular, the fact the solution is convex for all $t$ implies
\begin{equation}\label{lambda pinching}
\lambda_{i}\leq   H\quad \text{on}\ M\times [0,T)
\quad \text{for each}\ i.
\end{equation}

Combining now \eqref{tildle lambda pinching} and \eqref{lambda pinching}, we conclude that
$M_{t}$ satisfies everywhere the pointwise curvature pinching estimate
\begin{equation}\label{pointwise pinching}
\lambda_{n}\leq  C_{3} \lambda_{1},
\end{equation}
where $C_{3}= \varepsilon^{-1}_{0}$.

Thus, a uniform double side bound for the inradius and outer radius of the evolving hypersurfaces
follows by combining \eqref{pointwise pinching} with Lemma \ref{radius estimates} and  Lemma \ref{mixed volume preservation}.
\begin{corollary}
\label{double bound of radius}
Under the assumptions of Theorem \ref{pinching}, there are constants $D_{i} = D_{i}(n, \beta, V_{n-m})$, $i \in \{1, 2\}$,  such that
\[
D_{1} \leq \rho_{-}(t)\leq \rho_{+}(t)  \leq D_{2}\qquad \text{ for every } \quad t\in [0,T)\,.
\]
\end{corollary}

Also as in \cite[Corollary 4.6.]{And94}
curvature pinching implies that the first derivatives of $F$ has  a uniform double side bound.
\begin{corollary}
\label{double bound of dof F}
Under the assumptions of Theorem \ref{pinching}, there are constants $C_i = C_i(n, M_{0})$, $i \in \{4, 5\}$,  such that
\[
C_{4} Id \leq \dot{F}(\mathscr{W})  \leq C_5 Id \qquad \text{ for every } \quad t\in [0,T)\,.
\]
\end{corollary}

\section{\bf Upper bound on $\Phi(F)$}\label{up bound for the powerF flow speed}

In this section uniform bounds from above on the speed for the flow
\eqref{mvpcf}--\eqref{barPhi} with $\Phi(F)$ given by \eqref{def_Phi} for some $\beta\geq1$
and for the curvature of the hypersurface are derived, depending only on the initial data.
Throughout the section, we assume that the flow satisfies the assumptions of Theorem \ref{pinching}.
As usual, we denote by $\Omega_t \subset \mathbb{R}^{n+1}$ the region enclosed by $M_t$.
The bounds on curvatures together with the
estimates in the next section will imply the long time existence of
the flow by well-known arguments.
Following the technical procedure in literature,
we first show that a ball with fixed center remains inside the evolving $\Omega_t$ for a suitable time interval.

\begin {lemma}\label{t0+smalltime}
If $B(p_{t_0},\rho_{t_0}) \subset \Omega_{t_0}$ for some $t_0 \in
[0,T)$, where $\rho_{t_0}={\rho_{-}}(t_0)$ is the inradius of
$M_{t_0}$, then there exists some constant $\tau=\tau( n, \beta, V_0)
>0$ such that $B(p_{t_0},\rho_{t_0}/2) \subset \Omega_t$ for every
$t \in [t_0, \min\{t_0+\tau, T\})$.
\end {lemma}
\begin{proof}
Proceeding similarly as in \cite[Lemma 8]{CR-M07} and \cite[Lemma 6.1]{McC05}, Our procedure is
to compare the deformation of $M_{t}$ by the equation \eqref{mvpcf}
with a sphere shrinking under a quasilinear flow with speed given by a multiple of the power ${\beta}$ of the mean curvature.

For convenience, let $r_B(t)$ be the radius at time $t$ of a sphere $\partial B(p_{t_0},r_B(t))$ centered at $p_{t_0}$,
evolving according to
\begin{equation}\label{ev.rB}
\ds \frac{\dif r_B(t)}{\dif t} = - \
\frac{C_{6}}{(r_B(t))^{\beta}}.
\end{equation}
with the initial condition $r_B(t_0) = \rho_{t_0}$ and a positive constant $C_{6}$ to be determined,
this ODE has
solution
\begin{equation*}
r_B(t)=\left( \rho_{t_0}^{\beta+1}-C_{6}(\beta+1)(t-t_{0})\right)^{\frac1{\beta+1}}.
\end{equation*}
Thus, the sphere shrinks to $B(p_{t_0},\rho_{t_0}/2)$
if and only if
\begin{equation*}
t-t_{0} \ = \
\frac{2^{\beta+1}-1}{C_{6}(\beta+1)}\left(\frac{\rho_{t_0}}{2}\right)^{\beta+1}=:\tau.
\end{equation*}

Now set $\mathfrak{X}(\cdot, t)=X(\cdot, t) - p_{t_0}$
and  $f(x,t) = |\mathfrak{X}(x,t)|^{2}- |r_B(t)|^{2}$, from \eqref{mvpcf}
and \eqref{ev.rB}, it follows
\begin{equation} \label{evol.f}
\partial_{t} f= \  2(\bar{\phi}(t)-\Phi(F)) \left\langle \nu_t,\mathfrak{X}\right\rangle + 2C_{6} \left(\frac1{r_{B}}\right)^{ \beta-1}.
\end{equation}
On the other hand, by the Euler's relation we compute
\begin{equation*}
\Delta_{\dot \Phi} f =F^{\beta-1}\Delta_{\dot F}  |\mathfrak{X}(x,t)|^{2} \\
= -2\Phi \left\langle \nu_t,\mathfrak{X}\right\rangle+2  F^{\beta-1}{\dot F}^{k}_{k}. \nonumber
\end{equation*}
Therefore, \eqref{evol.f} can be rewritten as
\begin{align*} 
\partial_{t} f = \Delta_{\dot \Phi} f
+2\bar{\phi}(t)\left\langle \nu_t,\mathfrak{X}\right\rangle
- 2 F^{\beta-1}{\dot F}^{k}_{k}
 + 2C_{6} \left(\frac1{r_{B}}\right)^{ \beta-1}.
\end{align*}
As strict convexity holds for each $M_t$, we have that $\bar{\phi}(t)$,
and $\left\langle \nu_t,\mathfrak{X}\right\rangle$ are all positive,
which  allow us to disregard the term containing $\bar{\phi}(t)$.
We note that at any point of $M_{t}$,
\begin{align*}
{\dot F}^{k}_{k}= \tr\left({\dot F}\right)\leq n C_{5},
\end{align*}
due to Corollary \ref{double bound of dof F}
and $f(\lambda_{1}, \cdots, \lambda_{n})\leq \lambda_{n} \leq \frac1{r_{B}}$.
So if we take $C_{6}= n C_{5} $,
then
\begin{align*}
\partial_{t} f \geq \Delta_{\dot \Phi} f.
\end{align*}
Thus, since $f(x,t_{0})\geq 0$ using
a standard comparison principle we conclude that
\begin{equation} \label{comparision}
f(x,t)\geq 0,
\end{equation}
and hence $B(p_{t_0},\rho_{t_0}/2) \subset \Omega_t$ for every
$t \in [t_0, \min\{t_0+\tau, T\})$.
which concludes the proof.
\end{proof}

In order to obtain an upper bound on $\Phi(F)$,
as in \cite{Tso,McC04,McC05,CR-M07,CS10},
the method is to study the evolution under the flow of the
function
\begin{equation*}
Z_{t}=\frac{\Phi(F)}{\mathcal {S}-\epsilon}.
\end{equation*}
Here $\mathcal {S}=\langle\nu, \mathfrak{X}\rangle$,
the position vector field $\mathfrak{X}(\cdot, t)=X(\cdot, t) - p_{t_0}$  (with origin $p_{t_0}$) on $M_t$,
and $\epsilon$ is a constant to be
chosen later.
Its evolution equation is straightforward to compute using
\eqref{ev support for the powerF} and  \eqref{ev Phi}.
\begin{corollary}
For $t \in [0, T)$ and any constant $\epsilon$,
\begin{align}\label{evZ}
\partial_{t}Z&=\Delta_{\dot \Phi} Z
+\frac{2\left\langle\nabla Z,\nabla\mathcal {S}\right\rangle_{\dot
\Phi}}{\mathcal {S}-\epsilon} -\epsilon\frac{Z}{\mathcal {S}-\epsilon}\tr_{\dot
\Phi}(A\mathscr{W}) +
(1+\beta)Z^{2}\\
 &\quad- \frac{\bar{\phi}}{\mathcal {S}-\epsilon}\left(\tr_{\dot
\Phi}(A\mathscr{W})\right)-\frac{Z}{\mathcal {S}-\epsilon}\bar{\phi}.\nonumber
 \end{align}
\end{corollary}

The above lemma assists us by allowing us to consider a uniform
bound on the speed of the flow.
\begin{corollary}
For $t \in [0,T)$,
\begin{equation}\label{bound on the speed}
\Phi(F)(\cdot,t) < C_{7}= C_{7}(n,\beta,M_{0}),
\end{equation}
moreover,
\begin{equation}\label{bound on the speed F}
F(\cdot,t) <C_{8}:= C^{1/\beta}_{7}.
\end{equation}
\end{corollary}
\begin{proof}
For any fixed $t_0 \in [0,T)$, let $p_{t_0}$ and $\rho_{t_{0}}$ be
as in Lemma \ref{t0+smalltime}.
Then, using the convexity of the $M_{t}$'s and taking the constant $\epsilon = \rho_{t_{0}}/4$ leads to
\begin{equation}\label{good epsilon}
\mathcal {S}-\epsilon\geq\epsilon>0,
\end{equation}
on the same time interval, which ensures
$Z_{t}=\frac{\;F\;}{\mathcal {S}-\epsilon}$ is well-defined.

Let us go back to equation \eqref{evZ}, since strict convexity
holds for each $M_{t}$, $\Phi(F)$, $\bar{\phi}$ and $\tr_{\dot
\Phi}(A\mathscr{W})$ are all positive, which together
with \eqref{good epsilon}, the term containing $\bar{\phi}$ can be neglected. Furthermore, note
that $\Phi(F)$ is homogeneous of degree $\beta$, Euler's relation and
\eqref{tildle lambda pinching} together give the following
\[
\tr_{\dot \Phi(F)}(A\mathscr{W})=\dot \Phi ^{i}\lambda^{2}_{i}
\geq
\varepsilon_{0}n F \dot \Phi ^{i}\lambda_{i}=n \varepsilon_{0}\beta F \Phi.
\]
Now from the above remark,
\begin{align}\label{ev.inq.Z}
\partial_{t}Z&\leq \Delta_{\dot \Phi} Z
+\frac{2\left\langle\nabla Z,\nabla\Phi\right\rangle_{\dot
\Phi}}{\mathcal {S}-\epsilon} -\epsilon\varepsilon_{0}\beta \Phi^{\frac1{\beta}} Z^{2} +
(1+\beta)Z^{2}.
 \end{align}
On the other hand, from \eqref{good epsilon}  it follows
\[
Z\leq \frac{\Phi}{\epsilon}.
\]
Applying this to \eqref{ev.inq.Z} gives
\begin{align*}
\partial_{t}Z&
\leq \Delta_{\dot F} Z +\frac{2\left\langle\nabla
Z,\nabla\Phi\right\rangle_{\dot F}}{\mathcal {S}-\epsilon}\ +\left(
(1+\beta)-\epsilon^{1+\frac{1}{\beta}}n\beta\varepsilon_{0}
Z^{\frac{1}{\beta}}\right)Z^{2}.
 \end{align*}
Assume that in $(\bar{x},\bar{t})$, $\bar{t} \in [t_0,
\min\{t_0+\tau, T\})$,  $Z$ attains a big maximum $C\gg 0$ for the
first time. Then
\[
Z(\bar{x},\bar{t})= C(\mathcal {S}-\epsilon)(\bar{x},\bar{t})\geq
\epsilon C,
\]
which gives a contradiction if
\[
C> \max_{x\in
M^{n}}\left\{Z(x,t_0),\frac{1}{\epsilon}\left(\frac{(\beta+1)}{n\varepsilon_{0}\epsilon
\beta}\right)^{\beta}\right\}.
\]
Thus,
\[
Z_{t}(x)\leq \max_{x\in
M^{n}}\left\{Z(x,t_0),\frac{1}{\epsilon}\left(\frac{(\beta+1)}{n\varepsilon_{0}\epsilon
\beta}\right)^{\beta}\right\},
\]
on $[t_0, \min\{t_0+\tau, T\})$.
Applying the fact that our flow preserves the mixed volume $V_{n-m}$ to Lemma \ref{radius estimates},
we can control the maximal distance between $p_{t_0}$ and the points in $\partial\Omega_{t}$ by a big positive constant $D$.
Then from the definition of $Z_{t}$, it follows
\[
\Phi(x,t)\leq \left(D-\epsilon\right)\max_{x\in
M^{n}}\left\{Z(x,t_0),\frac{1}{\epsilon}\left(\frac{\cc_{\kappa}(D_{2})(\beta+1)}{n\varepsilon_{0}\epsilon
\beta}\right)^{\beta}\right\},
\]
on $[t_0, \min\{t_0+\tau, T\})$. Since $t_0$ is arbitrary, and
$\tau$ does not depend on $t_0$, this implies
\begin{align*}
\Phi(x,t)\leq& \left(D-\epsilon\right)\max_{x\in
M^{n}}\left\{Z(x,t_0),\frac{1}{\epsilon}\left(\frac{\cc_{\kappa}(D_{2})(\beta+1)}{n\varepsilon_{0}\epsilon
\beta}\right)^{\beta}\right\}\\
&=:C_{7}(n,\beta,M_{0}),
\end{align*}
on $[0, T)$, and so \eqref{bound on the speed F} follows by the definition
of $\Phi(F)$.
\end{proof}

Inserting the estimate of \eqref{bound on the speed} into \eqref{barPhi} immediately gives the following

\begin{corollary}
For $t \in [0,T)$,
\begin{equation}\label{bound on the speed bar phi}
\bar{\phi}(t) < C_{6}.
\end{equation}
\end{corollary}

Hence the speed of the evolving hypersurfaces is bounded.
\begin{corollary}
For $t \in [0,T)$,
\begin{equation*}
\left|\frac{\partial}{\partial t}\mathrm{X}\left(p,t\right)\right| <
C_{9}:=2C_{6}.
\end{equation*}
\end{corollary}

The curvature of $M_{t}$ also remains bounded.

\begin{corollary}
For $t \in [0,T)$,
\begin{equation}\label{bound on the speed H}
\big|\mathscr{W}\big|< H\leq C_{10}.
\end{equation}
\end{corollary}
\begin{proof}
In the case of concave $F$, the homogeneity of $\Phi(F)$, Corollary \ref{double bound of dof F} and \eqref{lambda pinching}  imply that

\begin{equation*}
\beta \Phi= \dot{\Phi}\lambda_{i} \geq \varepsilon_{0}H \tr(\dot \Phi)
 =\varepsilon_{0}H F^{\beta-1}\tr(\dot F)\geq
\varepsilon_{0}H   \beta \Phi^{1 -  \frac{1}{\beta}}c_{5}.
\end{equation*}
Thus, by \eqref{bound on the speed}
\[
H\leq \frac{1}{\varepsilon_{0}C_{4}}\Phi^{\frac{1}{\beta}}\leq
\frac{1}{\varepsilon_{0}c_{5}}C_{6}^{\frac{1}{\beta}}=:C_{10},
\]
and so with the convexity of $M_{t}$
\[
\big|\mathscr{W}\big|<C_{10}.
\]
In the case of concave $F$, in view of the relation $F\geq \frac{1}{n}H$, we can easily obtain the estimate \eqref{bound on the speed H}.
\end{proof}

\section{\bf Long time existence}\label{longtime for the powerF flow}
In this section, it is shown that solution of the initial value
problem \eqref{mvpcf}--\eqref{barPhi} with $\Phi(F)$ given by \eqref{def_Phi} for some $\beta\geq1$,
with one of the following pinching conditions:
\begin{enumerate}[label={(\roman*).}, ref={(\roman*)}]
\item {for convex $F$ },
\begin{equation*}
 K(p)> \mathcal{C}F^n>0,
\end{equation*}
\item {for convex $F$ },
\begin{equation*}
 K(p)> \mathcal{C}H^n>0,
\end{equation*}
\end{enumerate}exists for all positive times. As usual, the first step is
to obtain suitable bounds on the solution on any finite time
interval $[0, T)$, which guarantees our problem has
a unique solution on the time interval such that the solution
converges to a smooth hypersurface $M_{T}$ as $t\rightarrow T$.
Thus, it is necessary to show that the solution remains uniformly
convex on the finite time interval which ensure the parabolicity
assumption of \eqref{mvpcf}.

First it is to show the preserving convexity of the evolving
hypersurface $M_{t}$. Recall that Theorem \ref{pinching} and Lemma
\ref{pinching imply convex lemma} together imply the uniform
convexity of $M_{t}$, however, comparing with the initial
assumptions of Theorem \ref{main result for powerF flow}, there is a priori
assumption $F>0$ in Theorem \ref{pinching}. As Cabezas-Rivas
and Sinestrari mentioned in \cite{CS10}, note that for small times
such an assumption holds due to the smoothness of the flow for small
times and the initial pinching condition \eqref{convex pinched} or  \eqref{concave pinched},
but it is possible that at some positive time all $\min K$,  $\min F$ and
$\min H$ tend to zero such that $ K/H^n$ and $ K/F^n$
remains bounded. Thus, to exclude such a behaviour, following
\cite{CS10}, it is necessary to complement Theorem \ref{pinching} by
establishing positivity of $F$ for the finite
time.

\begin {lemma} \label{nonegative of F} Under the hypotheses of Theorem \ref{pinching},
\begin{equation*} 
F(\cdot,t)\geq F(\cdot,0) e^{ -C_{11} T } \qquad \text{for all times}\
t \in [0, T). \end{equation*}
 \end {lemma}
\begin{proof}
Let us go back to the evolution
equation \eqref{evF} of $\Phi(F)$,
\begin{align*}
\partial_{t}\Phi = \Delta_{\dot \Phi} \Phi + (\Phi - \bar{\phi}) \,\tr_{\dot
\Phi} (A \mathscr{W}).
\end{align*}
Since under the hypotheses of Theorem \ref{pinching}, the evolving
hypersurfaces $M_t$ remains convex for every $t \in [0,T)$, i.e. $\lambda_i\geq 0$ on $[0,T)$,
taking a normal coordinate system at a point where $\mathscr{W}$ is diagonal,
we have
\begin{align}
\tr_{\dot \Phi} (A \mathscr{W}) =
        \sum_{i=1}^n \dot{\Phi}^i \lambda_i^2  &\geq 0.\notag
    \end{align}
Then we have the following computation of $\Phi$:
\begin{equation}\label{estimate of Phi}
\partial_{t}\Phi \geq \Delta_{\dot \Phi} \Phi  - \bar{\phi} \sum_{i=1}^n \dot{\Phi}^i \lambda_i^2
\geq \Delta_{\dot \Phi} \Phi  -C_{11}\sum_{i=1}^n \dot{\Phi}^i\lambda_i,
\end{equation}
where we have used the estimates \eqref{bound on the speed} and \eqref{bound on the speed bar phi} in the last step,
and $C_{11}=C_{6}(C_{10})$.

Now we use the fact that $\Phi$ is  homogeneous of degree $ \beta$,
we have the evolving inequality of $\Phi$
\[
\partial_{t}\Phi \geq \Delta_{\dot \Phi} \Phi  -C_{11}\beta {\Phi},
\]
The parabolic maximum principle and the fact $\Phi(\cdot, 0) \geq 0 $ now give
    \[\Phi(\cdot,t)\geq\Phi(\cdot,0) e^{ -C_{11} \beta t }. \]
This also implies that for all times $t \in [0, T)$,
\[F(\cdot,t)\geq F(\cdot,0) e^{ -C_{11} T }. \]
\end{proof}

\begin{corollary}
 \label{positive tildelambda}The solution of \eqref{mvpcf}--\eqref{barPhi} with $\Phi(F)$ given by \eqref{def_Phi} for some $\beta\geq1$,
with one of the following pinching conditions:
\begin{enumerate}[label={(\roman*).}, ref={(\roman*)}]
\item {for convex $F$ },
\begin{equation*}
 K(p)> \mathcal{C}F^n>0,
\end{equation*}
\item {for convex $F$ },
\begin{equation*}
 K(p)> \mathcal{C}H^n>0,
\end{equation*}
\end{enumerate} are uniformly convex on any finite time
interval; that is, for any $T < +\infty$, $T \leq T_{\max}$, we have
$$
\inf_{M \times [0,T)} \lambda_i >0, \qquad \forall
i=1,\dots,n.
$$
Therefore, Theorem \ref{pinching} is valid also without the
hypothesis that $F>0$ for $t \in (0,T)$. The same holds for
the other results that have been obtained until here under the same
assumptions of Theorem \ref{pinching}.
\end{corollary}


Now to obtain uniform estimates on all orders of curvature derivatives and
hence smoothness and convergence of the $M_{t}$ for the flow,
\eqref{mvpcf}--\eqref{barPhi} with $\Phi(F)$ given by \eqref{def_Phi}
 we use a more PDE theoretic approach,
following an argument similar to the one in \cite{CS10}.

Since the uniformly bound the speed of our flow we have just obtained implies a bound on all principal curvatures,
as in \cite{CS10} ( see also \cite [Lemma 3.4] {Sch05}),
in view of Lemma \ref{t0+smalltime}, one can adopt a local graph representation for a convex hypersurface with uniformly bounded $C^2$ norm.

\begin{corollary}
 \label{graph_u}
 There exists $r$ (depending only on $\max F$) with the following property.
Given any $(p_{t_{0}}, t_{0}) \in M \times \,(0, T)$, there is a neighborhood $\mathcal {U}$ of the point $x_{0}:= X(p_{t_{0}}, t_{0})$ such that $M_t \cap \mathcal {U}$ coincides with the graph of a smooth function
$$u: B_r \times J \rightarrow \mathbb{R}, \qquad \mbox{ for all } t \in J.$$
Here $B_r \subset T_{p_{t_{0}}} M_{t_{0}}$ is the ball of radius $r$ centered at $x_{0}$ in the hyperplane tangent to $ M_{t_{0}}$ at $x_{0}$,
and $J$ is the time interval $J=[t_0, \min\{t_0+\tau, T\})$
(near to $t_0$), $p_{t_0} \in \Omega_t$.
In addition, the $C^2$ norm of $u$ is uniformly bounded (by a constant depending only on $\max F$).
\end{corollary}


In graphical coordinates, we see that
\begin{equation} \label{parametrisation_u}
X(p,t)=\left(x(p,t), u(x(p,t),t)\right).
\end{equation}
So the metric and its inverse (cf. \cite{CS10}) are given by
\begin{equation*}
g_{ij} = \delta_{ij} + D_i u \, D_j u, \quad  \quad \quad g^{ij} = \delta^{ij} - \frac{D^i u \, D^j u}{1 + |D u|^2},
\end{equation*}
where $D_i$ denote the derivatives with respect to these local coordinates,
respectively.
The outward unit normal vector of
$M_{ t}$ can be expressed as
\begin{equation*}
\nu=\frac{1}{\left|\xi\right|}\Big( -D u,1\Big)
\end{equation*}
with
\begin{equation*}
\big|\xi\big|=\sqrt{1+\left|D u\right|^{2}}.
\end{equation*}

\begin{equation*} 
h_{ij} =\frac{D_{ij} u}{(1 + |Du|^2)^{1/2}},
\end{equation*}
 and
 \begin{equation*} 
h_{j}^{i} =\left( \delta^{ik} - \frac{D^i u \, D^k u}{1 + |D u|^2}\right) \frac{D_{kj} u}{(1 + |Du|^2)^{1/2}},
\end{equation*}

In addition, the Christoffel symbols have the expression:
\begin{equation*}
\Gamma_{ij}^k = \left(\delta^{kl} - \frac{D^k u \, D^l u}{1 + |D u|^2}\right) D_{ij} u \, D_l u.
\end{equation*}

\begin{theorem}
 \label{Calpha estimate}
Under the hypotheses of Theorem \ref{main result for powerF flow}, for any $0<t_{0}<T$, $\alpha
\in (0,1)$ and every natural number $k$, there exist constant
$C_{14}$, depending on $(n,\beta, M_{0},k)$ such that
\begin{align}
\label{u estimate C k}\|u\|_{C^{k}(M \times (t_{0},T])} \leq C_{14}.
\end{align}
\end{theorem}
\begin{proof}
Due to a tangential diffeomorphism into the flow \eqref{mvpcf},
$u$ then satisfies the following parabolic PDE
\begin{align}\label{ev.param.u}
\partial_{t}u=\big|\xi\big|(\Phi_t -\bar{\phi}_t).
\end{align}
It is easy to see that along the flow \eqref{mvpcf},
in the coordinate system the corresponding evolution equations for $u$ and $\Phi$,
at $(x,t) \in B_r \times J,$
can be expressed as
\begin{equation}\label{paramu2 for PowerF}
\partial_{t}u= F^{\beta -1} g^{ik}\dot{F}_{ij}(\mathcal {W}) D_{k}D_{j} u -\big|\xi\big|\bar{\phi}_t
\end{equation}
and
 \begin{align}\label{locally ev.Phi}
 \partial_{t}\Phi =  F^{\beta -1} g^{ik}\dot{F}_{ij}(\mathcal {W})D_{k}D_{j} \Phi -& F^{\beta -1} g^{ik}\dot{F}_{ij}(\mathcal {W}) \Gamma_{kj}^lD_{jl}\Phi \\
 &+(\Phi_t -\bar{\phi}_t) \, F^{\beta -1} g^{ik} g^{lm}\dot{F}_{ij}(\mathcal {W})h_{il} h_{jm} ,
\nonumber
 \end{align}
respectively.
 Since we have uniform bounds on the curvatures both from above and from below on any finite time interval,
 in view of Corollary \ref{double bound of radius}, Corollary \ref{double bound of dof F},
 equations \eqref{bound on the speed} and  \eqref{bound on the speed bar phi},
 equations \eqref{paramu2 for PowerF} and \eqref{locally ev.Phi} are uniformly parabolic with uniformly bounded coefficients.
Then applying Theorem \ref{krylov-safonov} to  \eqref{locally ev.Phi}
gives that for any $0 <r' < 1$ and $J'\subset J$ there exist some
constants $C_{12}>0$ and $\alpha \in (0,1)$ depending on $n, m,
\beta, M_{0}$ such that
\begin{align}
\label{Phi estimate C alpha}\|\Phi\|_{C^{\alpha}(B_{r'} \times J')} &\leq C_{12},\\
\label{u estimate C alpha}\|u\|_{C^{\alpha}(B_{r'} \times J')} &\leq
C_{12}.
\end{align}
Now for any fixed time $t_{1}\in J'$, we want to follow an argumentation similar to the one used in \cite[Lemma 6.3]{CS10} ( also as in \cite{McC05}).
Then in the case of cancave $F$, the Bellman's extension $\bar{F}$ of $F$
is given by
\[
 \bar{F}(\bar{\lambda}):= \inf_{\lambda\in
\hat{\Gamma}}\left[F\left(\lambda\right)+
DF\left(\lambda\right)\left(\bar{\lambda}-\lambda\right) \right],
\]
for any $\bar{\lambda} \in \Gamma_+$,
while in the case of convex $F$,
 $\bar{F}$ takes the form
\[
 \bar{F}(\bar{\lambda}):= \sup_{\lambda\in
\hat{\Gamma}}\left[F\left(\lambda\right)+
DF\left(\lambda\right)\left(\bar{\lambda}-\lambda\right) \right].
\]
Notice that $F$ is
homogeneous of degree one, the extension simplifies to
\[
 \bar{F}(\bar{\lambda})=\left\{ \begin{aligned}\inf_{\lambda\in \hat{\Gamma}}
DF\left(\lambda\right)\bar{\lambda}\quad \text{for $F$ concave},\\
\sup_{\lambda\in \hat{\Gamma}}
DF\left(\lambda\right)\bar{\lambda}\quad \text{for $F$ convex }.
\end{aligned}
 \right.
\]
Since the Bellman extension is the infimum or the supremum of linear functions,
 it preserves concavity or concavity, by definition and
homogeneity.
Importantly, $ \bar{F}$  coincides with $F$ on $\hat{\Gamma}$
by homogeneity of $F$. Furthermore, using the definition of
$ \bar{F}$ and Corollary \ref{double bound of dof F}, $ \bar{F}$ is
uniformly elliptic, that is
$$
C_5 |\bar{\eta}|\leq  \bar{F}(\bar{\lambda}+\bar{\eta}) -
 \bar{F}(\bar{\lambda})\leq \sqrt{n}C_{5} |\bar{\eta}|, \quad
\mbox{for all}\; \bar{\lambda},\bar{\eta} \in {\mathbb{R}}^{n},\;
\bar{\eta}\geq 0.
$$

Let $V_{t_{1}}: B_{r'}\rightarrow \mathbb{R}$ be given by
\[
V_{t_{1}}(x)=\frac{\partial_{t}u(x,t_{1})}{\sqrt{1+\left|D u(x,t_{1})\right|^{2}}}.
\]
In view of equations \eqref{ev.param.u} and \eqref{Phi estimate C alpha},
we also have
\[
\|V_{t_{1}}\|_{C^{\alpha}(B_{r'})} \leq C_{12}.
\]
Thus, we have the correspoinding elliptic equation for $u$
\[
 \bar{F}^{\beta}(D^2u(x,t_{1}),u(x,t_{1}))=V_{t_{1}}(x)
\]
where
\[
V_{t_{1}}(x)=\left(\Phi(x,t_{1})-\bar{\phi}(t_{1})\right),
\]
satisfies the conditions of Theorem 3 from \cite{Caf89}.
Thus, this theorem gives that $||u||_{C^{2,\alpha}(B_r')} \leq C_{13}.$
Therefore, covering $M_{t}$ by finite many graphs on balls of radius
$r'$ implies that
$$
||u||_{C^{2,\alpha}(M)} \leq C_{13}.
$$
From this estimate on $u(\cdot,t)$ for any fixed $t_{1}$, following the
procedure of
 \cite[Theorem 2.4]{Tsai}
to equation \eqref{ev.param.u}, a $C^{2,\alpha}$ estimate for $u$
with respect to both space and time can be obtained. Once
$C^{2,\alpha}$ regularity is established, combinigng standard parabolic theory with  our short time existence result
yields uniform $C^k$ estimates \eqref{u estimate C k} for any
integer $k > 2$, which implies uniform $C^k$ estimates \eqref{u
estimate C k} for any integer $k\geq 1$ with the fact that $u$
and its first order derivatives are uniformly bounded.
\end{proof}

\begin{theorem} \label{long time exist}
Under the hypotheses of Theorem \ref{main result for powerF flow}, the solution $M_{t}$ of
\eqref{mvpcf} with initial condition $X_0 $, exists, is smooth
and satisfies at every point the initial pinching condition  on $[0, \infty)$.
\end{theorem}
\begin{proof}
As we know, the preserving pinching condition and smoothness throughout the interval of existence follows form
Theorem \ref{pinching}, and Lemma \ref{nonegative of F}.

On the other hand, it is clear that from the expression \eqref{parametrisation_u}
that all the higher order derivatives of
$ X_t$ are bounded if and only if the corresponding derivatives
of $u$ are bounded. Thus, uniform $C^k$ estimate \eqref{u estimate C k} of $u$ implies that all the derivatives of $ X_t$ are
also uniformly bounded.

It remains to show that the interval of existence is infinite.
Suppose to the contrary there is a maximal finite time $T$ beyond
which the solution cannot be extended.
Then the evolution equation \eqref{mvpcf} implies that
\[
\|\mathrm{X(p,\sigma)}-\mathrm{X(p,\tau)}\|_{C^{0}(X_{0})}\leq
\int_{\tau}^{\sigma}\big| \bar{\phi} - \Phi\big|\left(p,t\right)\dif t
\]
for $0\leq \tau \leq \sigma < T$. By \eqref{bound on the speed} and
\eqref{bound on the speed bar phi}, $\mathrm{X}(\cdot,t)$ tends to a
unique continuous limit $\mathrm{X}(\cdot,T)$ as $t\rightarrow T$.
In order to conclude that $\mathrm{X}(\cdot,T)$ represents a
hypersurface $M_{T}$, next under this assumption and in view of the
evolution equation \eqref {evmetric for the powerF} the induced metric $g$ remains
comparable to a fix smooth metric $\tilde{g}$ on $M^{n}$:
\[
\left|\frac{\partial}{\partial
t}\left(\frac{g(u,u)}{\tilde{g}(u,u)}\right)\right|
=\left|\frac{\partial_{t}
g(u,u)}{g(u,u)}\frac{g(u,u)}{\tilde{g}(u,u)}\right| \leq
2|\Phi(x,t_{1})-\bar{\phi}||A|_{g}\frac{g(u,u)}{\tilde{g}(u,u)},
\]
for any non-zero vector $u\in T M^{n}$,  so that ratio of lengths is
controlled above and below by exponential functions of time, and
hence since the time interval is bounded, there exists a positive
constant $C$ such that
\[
\frac{1}{C}\tilde{g}\leq g\leq C\tilde{g}.
\]
Then the metrics $g(t)$ for all different times are equivalent, and
they converge as $t\rightarrow T$ uniformly to a positive definite
metric tensor $g(T)$ which is continuous and also equivalent by
following Hamilton's ideas in \cite{Ham82}. Therefore using the
smoothness of the hypersurfaces $M_{t}$ and interpolation,
\allowbreak
\begin{align*}
&\|\mathrm{X(p,\sigma)}-\mathrm{X(p,\tau)}\|_{C^{k}(X_{0})} \\
&\leq C
\|\mathrm{X(p,\sigma)}-\mathrm{X(p,\tau)}\|^{1/2}_{C^{0}(X_{0})}
\|\mathrm{X(p,\sigma)}-\mathrm{X(p,\tau)}\|^{1/2}_{C^{2k}(X_{0})}\\
&\leq C |\sigma - \tau|^{1/2},
\end{align*}
so the sequence $\{\mathrm{X(t)}\}$ is a Cauchy sequence in
$C^{k}(X_{0})$ for any $k$. Therefore $M_{t}$ converges to a smooth
limit hypersurface $M_{T}$ which must be a compact embedded
hypersurface in $\mathbb{R}^{n+1}$. Finally, applying the
local existence result, the solution $M_{T}$ can be extended for a
short time beyond $T$, since there is a solution with initial
condition $M_{T}$, contradicting the maximality of $T$. This
completes the proof of Theorem \ref{long time exist}.
\end{proof}

\section{\bf Exponential convergence to the sphere}\label{convergence for the powerF flow}
Observe that, to finish the proof of Theorem \ref{main result for powerF flow}, it
remains to deal with the issues related to the convergence of the
flow \eqref{mvpcf}: It is proved that solutions of equation
\eqref{mvpcf} with initial conditions converge, exponentially in the $C^{\infty}$-topology,
to a sphere in $\mathbb{R}^{n+1}$ as $t$ approaches infinity, whether $F$ be convex or concave.

Since our previous results depend on the lower bound for $F$ in time $T$,
we know that the solution is smooth and uniformly convex on any finite time interval,
but we cannot completely control its behavior as $t \to \infty$.
So at this stage we cannot exclude that $\min_{M_t} F \to 0$ as $t \to \infty$ so that uniform convexity is lost.
Since the regularity estimates depend on uniform convexity,
some additional argument is needed to ensure the existence of a smooth limit as $t \to \infty$.

We first want to show that, if a smooth limit exists,
the solution of \eqref{mvpcf} has to be a round sphere;
the existence of the limit will be proved afterwards.
To address the first step,
as in Theorem \ref{pinching} we consider the quotient $Q_{2}$, that is $K/F^n$, for convex $F$ and
the quotient $Q_{1}$, that is  $K/H^n$, for concave $F$.

To understand the behaviour of the quotients $Q_{i}$, $i=1,2$, from
their evolution equations \eqref{evtildeq} and \eqref{evtildeq2}
 by application of the maximum principle,
previously we need to check that the global term $ \bar{\phi}(t)$
has the right sign.

\begin {lemma}\label{a key estimate2}
There exists a constant $\bar{\phi}_0 =\bar{\phi}_0(n, m, \beta,) > 0$ such that $\bar{\phi}(t) \geq \bar{\phi}_0$ for all $t \geq 0$,
whether $F$ be convex or concave.
\end {lemma}
\begin{proof}
In view of estimates \eqref{tildle lambda pinching}, we can use Lemma \ref{property Em} $iii)$ to estimate
\begin{align*}
\left(\,\int_M E_{m+1} \dif \mu_{t}\right) ^{\frac{\beta+m+1}{m+1}}& \leq |M_t|^{\frac{\beta}{m +1}}\int_M \left(E_{m+1}\right) ^{\frac{\beta+m+1}{m+1}} \, \dif \mu_{t}\\
&{\leq}
 |M_t|^{\frac{\beta}{m +1}} \int_M \left(n H\right)^\beta E_{m+1}\, \dif \mu_{t} \\
 &\! \!\os{\eqref{tildle lambda pinching}}{\leq} \varepsilon_{0}^{-\beta} |M_t|^{\frac{\beta}{m +1}} \int_M \left(\lambda_{1}\right)^\beta E_{m+1}\, \dif \mu_{t} \\
&\leq
 \varepsilon_{0}^{-\beta}  |M_t|^{\frac{\beta}{m +1}} \int_M \Phi(F) E_{m+1}\, \dif \mu_{t}
\end{align*}
where we have applied a H\"older inequality in the first inequality.
Hence
\begin{align} \label{estimate_phi below}
\bar{\phi}(t)=  &\frac{\int_M \Phi(F) E_{m+1} \dif \mu_{t}}{\int_M E_{m+1} \dif \mu_{t}} \notag\\
 &\geq \varepsilon_{0}^{\beta} \left(\frac{\int_M E_{m+1} \dif \mu_{t}}{|M_t|}\right)^{\frac{ \beta}{m +1}}
=\left(\frac{ V_{n-m-1}}{ V_{n}}\right)^{\frac{ \beta}{m +1}}.
\end{align}

Now, the fact the mixed volume $V_{n-m}$ is fixed along the flow and Corollary \ref{double bound of radius} implies that the area $V_{n}$ of $M_t$ is not greater than the area of a sphere of radius $c_2$, while the mixed volume $V_{n-m-1}$ of $M_t$ is greater than the $V_{n-m-1}$ of a sphere of radius $D_1$. The use of this in the expression on the right hand side of \eqref{estimate_phi below} gives the desired lower bound for $\bar{\phi}$.
\end{proof}

Now we define $F_{\min}(t) = \min_M F(\, \cdot \,, t)$.
In the next lemma we will show that $F_{\min}$ does not decay too fast, and this will be enough for our purposes.
\begin {proposition} \label{integral min F}
We have
\begin{equation*}
\int_{0}^{\infty} F_{\min}(t) \dif t =+\infty,
\end{equation*}
whether $F$ be convex or concave.
\end {proposition}
\begin{proof}
Let us first estimate  $\mathscr{P}(t) := \min_M \Phi(\, \cdot \,, t)$.
Since our evolving manifolds are convex,
equation \eqref{estimate of Phi} gives
\begin{align*}
\left(\partial_{t}-\Delta_{\dot \Phi} \right)\Phi  &\geq   - \bar{\phi} \sum_{i=1}^n \dot{\Phi}^i \lambda_i^2
\os{\eqref{bound on the speed bar phi}}{\geq} - C_{7} \sum_{i=1}^n \dot{\Phi}^i \lambda_i^2\\
&\os{\eqref{tildle lambda pinching}}{\geq} - C_{7}\lambda_n\sum_{i=1}^n \dot{\Phi}^i \lambda_i
\os{\eqref{pointwise pinching}}{\geq} - C_{7}C_{3}F\sum_{i=1}^n \dot{\Phi}^i \lambda_i\\
&= - C_{7}C_{3}\beta{\Phi}^{1 + \frac{1}{\beta}}=:-\tilde{C}{\Phi}^{1 + \frac{1}{\beta}},
\end{align*}
where we have used homogeneity of $ \Phi$ and normalisation of $F$.
Therefore by applying the maximum principle to the above inequality we have
$$\mathscr{P} \geq  \left(\mathscr{P}(0)^{-\frac{1}{\beta}} + \tilde {C} \, \frac{t}{\beta}\right)^{-\beta},$$
from which it follows that
 $$\int_0^\infty F_{\min} (t) \, \dif t = \lim_{s \to \infty} \int_0^s n \,\mathscr{P}^{\frac{1}{\beta}} (t) \, \dif t
 \geq \lim_{s \to \infty}  \int_0^s n \left(\mathscr{P}(0)^{-\frac{1}{\beta}} + \tilde {C} \, \frac{t}{\beta}\right)^{-1} = \infty.$$
\end{proof}
\begin {remark} \label{integral min H} Since $H\geq n F$ in the case of concave $F$,  we have
\begin{equation*}
\int_{0}^{\infty} H_{\min}(t) \dif t =+\infty.
\end{equation*}
\end {remark}

The following result shows that if the limiting hypersurfaces exists, it has to be a round sphere.
\begin{proposition} \label{the shape of M}
The shape of $M_{t}$ approaches the shape of a round sphere as $t\rightarrow \infty$.
\end{proposition}
\begin{proof}
In the case of convex $F$, in the proof of Theorem \ref{pinching},
by applying the weak maximum principle to  \eqref{evtildeq estimate} for
$\mathcal {Q}(t)=\min_{M_{t}} \left(\frac{K}{F^n} \right)$, we have
$\mathcal {Q}(t) \geq \mathcal {Q}(0)$ .
Now to analysis the shape of $M_{t}$ as $t\rightarrow \infty$,
we apply the maximum principle to the evolution equation \eqref{evtildeq2}.
Suppose $Q_{2}$ attained a new minimum at some
$\left(p_{0}, t_{0}\right)$, $ t_{0}>0$.
the strong maximum principle then implies that $Q_{2}$ is identically constant.
If this is the case, from equation \eqref{evtildeq2} we must have
\begin{align*}
0\equiv &\frac{1}{ F^{2}}\left| F\nabla\mathscr{W}
-\mathscr{W}\nabla F\right|^{2}_{\dot \Phi, b}
+ \tr_{ b-\frac{n}{ F}\dot F}\left(
\ddot{\Phi}(\nabla\mathscr{W} ,\nabla \mathscr{W})\right)\\
 &+n\frac{ 1}{F}\ddot{F}(\nabla \mathscr{W} ,\nabla\mathscr{W})
+\big[( \beta-1) \Phi+\bar{\phi} \big]
\frac{ 1}{ F}\left(n\tr_{\dot F}( A\mathscr{W} ) - H F\right).
\end{align*}

Since each of the sum of the first two terms, the third term and the last term are nonnegative,
each must be identically equal to zero.
Thus, from the fact that $\bar{\phi}$ has a positive lower bound, $M_{t}$ is convex and $\beta \geq 1$,
the last expression in brackets must equal to zero, this means that at any point of $M_{t}$,
\begin{align*}
0&\equiv n\tr_{\dot F}( A\mathscr{W} ) - H F\\
&\qquad= \sum_{ij}\left(\dot{f}_{i}\lambda_{i}^{2}-\lambda_{i}\dot{f}_{j}\lambda_{j}\right)
=\sum_{i\neq j}\left(\dot{f}_{i}\lambda_{i}-\dot{f}_{j}\lambda_{j}\right)
\left(\lambda_{i}-\lambda_{j}\right).
\end{align*}
Now for any pair $i\neq j$ with $\lambda_{i}\neq \lambda_{j}$,
Lemma \ref{relation of dotF and dotf} implies
\[\left(\dot{f}_{i}\lambda_{i}-\dot{f}_{j}\lambda_{j}\right)
\left(\lambda_{i}-\lambda_{j}\right)\neq 0.\]
Thus,
\[\sum_{i\neq j}\left(\dot{f}_{i}\lambda_{i}-\dot{f}_{j}\lambda_{j}\right)
\left(\lambda_{i}-\lambda_{j}\right)\neq 0.\]
So we must have $\lambda_{i}=\lambda_{j}$ for all $i$ and $j$.
Thus,$M_{t_{0}}$ is umbilical everywhere, and therefore is a round sphere.

In the case of concave $F$,
a similar strong maximum principle argument on the evolution equation \eqref{evtildeq} as above
again shows that the shape of $M_{t}$ as $t\rightarrow \infty$ is a round sphere.
\end{proof}

Since the speed for the flow has a homogeneity degree larger than
one in the curvatures, in contrast with the standard Laplacian $\Delta$,
the operators $\Delta_{\dot \Phi} =\beta {F}^{\beta-1}\Delta_{\dot F} $
which appear in the evolution equations become degenerate if the curvatures approach zero
as $t\rightarrow \infty$.
To exclude such a possibility,
unlike the previous approach for various geometric flow with higher homogeneity speeds
by using an interior estimate on porous medium equations,
we use an argument similar to Andrews and McCoy \cite{And12}
 to obtain a positive lower speed bounds on finite time intervals, after a sufficient big time.
\begin{proposition} \label{lower speed bound}
Set $\varpi_{t_{0}}(t) = \int_{t_{0}}^{t} \bar{\phi}(s)\dif s$.
Under the flow \eqref{mvpcf}, for any  $t_{0}$, there exists  a positive constant $\gamma$
such that
$$
\Phi \geq \frac{\rho_{-}+\varpi_{t_{0}}-R_{t_{0}}(t)}{\gamma(t-t_{0})},
$$
where $R_{t_{0}}(t)$ is the radius of a ball that encloses $M_{t_{0}}$ and
evolves under the flow \eqref{mvpcf} with the same $\bar{\phi}$ as that of the evolving hypersurface $M_{t}$.
In particular, Choosing $t_{0}$ big enough, we have a positive lower bound $C_{16}$ for $\Phi$ after a given waiting time.
\end{proposition}
\begin{proof}
We begin with the following observation, which originate from the work of Smoczyk \cite[Proposition 4]{Smo98}
for the mean curvature flow.
\begin {lemma}
If $M_{t_{0}}$ enclosed $p \in \mathbb{R}$, then under the flow \eqref{mvpcf}  for any $t_{0}$ there exists  a positive constant $\gamma$
such that for a short time later
$$\langle X-p, \nu\rangle+ \gamma(t-t_{0})\Phi -\varpi_{t_{0}}>0.$$
\end {lemma}
\begin{proof}[Proof of the lemma]
From \eqref{ev support for the powerF} and \eqref{ev Phi},
we have the evolution equation
\begin{align}\label{evolution quality for obtainning speed}
&\partial_{t}\left[\langle X-p, \nu\rangle+ \gamma(t-t_{0})\Phi -\varpi_{t_{0}}\right]\\
\notag
 &=\Delta_{\dot{\Phi}} \left[\langle X-p, \nu\rangle+ \gamma(t-t_{0})\Phi -\varpi_{t_{0}}\right]\\
 \notag
&\quad+\left[\langle X-p, \nu\rangle+ \gamma(t-t_{0})\Phi -\varpi_{t_{0}}\right] \,\tr_{\dot \Phi} (A \mathscr{W})\\
&\quad+\left[\varpi_{t_{0}}- \gamma(t-t_{0})\bar{\phi}\right] \,\tr_{\dot \Phi} (A \mathscr{W})
 +\left[\gamma-(\beta+1)\right]\Phi\notag.
\end{align}

Since
\begin{align*}
\tr_{\dot \Phi} (A \mathscr{W})&= \sum_{i=1}^n \dot{\Phi}^i \lambda_i^2{\leq} \lambda_n\sum_{i=1}^n \dot{\Phi}^i \lambda_i
\os{\eqref{pointwise pinching}}{\leq}C_{3}F\sum_{i=1}^n \dot{\Phi}^i \lambda_i= C_{3}\beta{\Phi}^{1 + \frac{1}{\beta}},
\end{align*}
we estimate
\begin{align*}
&\left[\varpi_{t_{0}}- \gamma(t-t_{0})\bar{\phi}\right] \,\tr_{\dot \Phi} (A \mathscr{W})
 +\left[\gamma-(\beta+1)\right]\Phi\\
 &\geq - \gamma(t-t_{0})\bar{\phi}\,\tr_{\dot \Phi} (A \mathscr{W})
 +\left[\gamma-(\beta+1)\right]\Phi
 \\
 &\geq - \gamma(t-t_{0})\bar{\phi}\,C_{3}\beta{\Phi}^{1 + \frac{1}{\beta}}
 +\left[\gamma-(\beta+1)\right]\Phi\\
  &\os{\eqref{bound on the speed F}, \eqref{bound on the speed bar phi}}{\geq}- \gamma(t-t_{0})\,C_{7}^{1+\frac{1}{\beta}}\,C_{3}\beta{\Phi}
 +\left[\gamma-(\beta+1)\right]\Phi
 \\
 &= \left[\left[1- (t-t_{0})\,C_{7}^{1+\frac{1}{\beta}}\,C_{3}\beta\right]\gamma
 -(\beta+1)\right]\Phi.
\end{align*}
So  taking $C_{15}= C_{7}^{1+\frac{1}{\beta}}\,C_{3}\beta $ and $\gamma=2(\beta+1)$,
we can deduce that this is positive for $t \in \left[t_{0}, t_{0}+ \frac{1}{2C_{15}}\right] $.
Thus, applying the maximum principle to \eqref{evolution quality for obtainning speed}
gives the result of the lemma.
\end{proof}
\noindent{\it Proof of the proposition, continued.}
For any given positive constant $\varepsilon$,
choose $t_{1}$ big enough to ensure $\lambda_n(M_{t}) \leq (1+\varepsilon) \lambda_1(M_{t})$ for
all $t \in \left[t_{1},\infty\right] $.
Fix $t_{0} \in \left[t_{0},\infty\right] $.
Choose $q\in \mathbb{R}$ to be the incentre of $M_{t_{0}}$ such that
$M_{t_{0}}$  encloses $B_{\rho_{-}}$.
Let $R_{t_{0}}(t_{0})= \rho_{+}$, which denotes the outer radius of $M_{t_{0}}$,
it follows that $B_{R_{t_{0}}(t_{0})}$ encloses $M_{t_{0}}$.
For a given $z$, choose $p$ to be the point in $M_{t_{0}}$ which maximizes $\langle p, \nu(z,t)\rangle $,
then we have $\langle p, \nu(z,t)\rangle \geq \rho_{-}$ for $t > t_{0}$.
Let $R_{t_{0}}^{+}(t)= \rho_{-}+ \varpi_{t_{0}}(t))$. Now we consider the two evolving spheres, whose radii satisfy
\begin{equation}\label{expanding sphere a}
\left\{
\begin{array}{ll}
\frac{\partial}{\partial t}R_{t_{0}}^{+}
    = \bar{\phi}(t),\\[2ex]
R_{t_{0}}^{+}(t_{0})= \rho_{-},
\end{array}\right.
\end{equation}
and
\begin{equation}\label{expanding sphere b}
\left\{
\begin{array}{ll}
\frac{\partial}{\partial t}R_{t_{0}}
    = \bar{\phi}(t)-R_{t_{0}}^{-\beta}(t),\\[2ex]
R_{t_{0}}(t_{0})= \rho_{+},
\end{array}\right.
\end{equation}
respectively.
Thus from \eqref{expanding sphere a} and \eqref{expanding sphere b},
we get
\begin{equation}\label{mirror of two spheres}
\left\{
\begin{array}{ll}
\frac{\partial}{\partial t}\left(R_{t_{0}}^{+}-R_{t_{0}}\right)
    = R_{t_{0}}^{-\beta}(t),\\[2ex]
R_{t_{0}}^{+}(t_{0})-R_{t_{0}}(t_{0})= \rho_{-}-\rho_{+},
\end{array}\right.
\end{equation}
Since $R_{t_{0}}$ is greater than the radius of the sphere
 with the same value of $V_{n-m}$ as the evolving hypersurface $M_{t}$ from \eqref{expanding sphere b},
 so $\bar{\phi}(t)>R_{t_{0}}^{-\beta}(t)$.
We deduce that $R_{t_{0}}$ is increasing on the time interval $\left[t_{0},t_{0}+ \frac{1}{2C_{15}}\right]$.
In particular, combining this with \eqref{bound on the speed bar phi},
we further have $R_{t_{0}}\leq \rho_{+} +\frac{C_{7}}{2C_{15}} $ by \eqref{bound on the speed bar phi}.
From \eqref{mirror of two spheres} this implies that
on $\left[t_{0},t_{0}+ \frac{1}{2C_{15}}\right]$,
we have
$$R_{t_{0}}^{+}(t)-R_{t_{0}}(t)\geq \rho_{-}-\rho_{+}
 + \left(\rho_{+} +\frac{C_{7}}{2C_{15}}\right)^{-\beta}\left(t-t_{0}\right),$$
which means that
$$\rho_{-}+\varpi_{t_{0}}-R_{t_{0}}(t)\geq \rho_{-}-\rho_{+} + \left(\rho_{+} +\frac{C_{7}}{2C_{15}}\right)^{-\beta}\left(t-t_{0}\right).$$
So if $\left(t-t_{0}\right) >  \left(\rho_{+} +\frac{C_{7}}{2C_{15}}\right)^{\beta} \left(\rho_{+}-\rho_{-}\right)$,
the quantity $\rho_{-}+\varpi_{t_{0}}-R_{t_{0}}(t)$ is positive.
In fact, using Proposition \ref{the shape of M} we may proceed exactly as in \cite[Theorem 3.1]{And12}
to show that $\left(t-t_{0}\right) >  \left(\rho_{+} +\frac{C_{7}}{2C_{15}}\right)^{\beta} \left(\rho_{+}-\rho_{-}\right)$ can be ensured
by  the pinching condition $\lambda_n(M_{t}) \leq (1+\varepsilon) \lambda_1(M_{t})$ for any small $\varepsilon$.
Now we apply the above lemma to estimate that for $t \in \left[t_{0}, t_{0}+ \frac{1}{2C_{15}}\right]$ and $\gamma=2(\beta+1)$
\begin{align*}
&\Phi \geq \frac{\langle p-X(z,t), \nu(z,t)\rangle +\varpi_{t_{0}}(t))}{\gamma(t-t_{0})}\\
& \geq \frac{\rho_{-}+\varpi_{t_{0}}-R_{t_{0}}(t)}{\gamma(t-t_{0})}\\
& \geq \frac{\rho_{-}-\rho_{+} + \left(\rho_{+} +\frac{C_{7}}{2C_{15}}\right)^{-\beta}\left(t-t_{0}\right)}{\gamma(t-t_{0})}
\end{align*}
Thus, our choice of any small $\varepsilon$  as required that $t_{1}$ is big enough
 can ensure that
$$\frac{1}{2C_{15}}>  \left(\rho_{+} +\frac{C_{7}}{2C_{15}}\right)^{-\beta}\left(t-t_{0}\right).$$
This imply a positive lower bound for $\Phi$ after a given waiting time, depending $n$, $\beta$ and $M_{0}$.
 \end{proof}

\begin {remark} \label{rew_evF}
We can follow an idea of  \cite{Sch06,CS10}  to show that the evolution equation for $F$,
which in a local coordinate system along \eqref{mvpcf} can be written as
\begin{equation*}
\partial_{t}F = D_i\left(F^{ij} D_j F^\beta\right)-\partial_{j}F^{ji} D_i F^\beta + \Gamma^j_{jl} F^{li} D_i F^\beta + (F^\beta - \bar{\phi}) \tr_{\dot F}( A\mathscr{W}),
\end{equation*}
where $D_i$ denote derivatives with respect to the coordinates.
This can be view as a porous medium equation as in \cite{Sch06,CS10}.
Unfortunately, we can not bound the coefficient $\Gamma^j_{jl}$ in this form and therefore an interior H\"older estimate
for solutions of such equations  established by DiBenedetto
and Friedman (\cite[Theorem\,1.2]{DF85}) can not applied to prove a uniform $C^{\alpha}$-esitmate for $F$.
\end {remark}

The next step is to show that the convergence to a sphere of $\mathbb{R}^{n+1}$ as $t\rightarrow \infty$ is exponential.
To address this, we consider the following function
\[
f= \frac{1}{n^n}-\frac{ K}{ H^n}.
\]
Then as remarked in Section \ref{Preserving pinching for powerF flow}, $f\geq 0$ with equality holding only at umbilic points,
which is the value assumed on a sphere. The following Lemma is an
immediate consequence of the evolution equation \eqref{evtildeq2} of
$ Q_{1}$.
\begin {lemma}
The quantity $f$ evolves under \eqref{mvpcf} satisfies
\begin{align}\label{evtildef}
\partial_{t} f=&\Delta_{\dot \Phi} f
+\frac{(n+1)}{n H^{n}}\left\langle\nabla
 f,\nabla H^{n}\right\rangle_{\dot \Phi}
-\frac{(n-1)}{n K}\left\langle\nabla  f,\nabla K\right\rangle_{\dot \Phi}
-\frac{ H^{n}}{n K}\left|\nabla f\right|_{\dot \Phi}^2
\notag\\
&-\Bigg(\frac{ Q_{1}}{ H^{2}}\left| H\nabla\mathscr{W}
-\mathscr{W}\nabla H\right|^{2}_{\dot \Phi, b}
+ Q_{1}\,\tr_{ b-\frac{n}{ H}Id}\left(
\ddot{\Phi}(\nabla\mathscr{W} ,\nabla \mathscr{W})\right)\Bigg)\\
&-\big[( \beta-1) \Phi+\bar{\phi} \big]
\frac{ Q_{1}}{ H}\left(n\bigl| A\bigl|^{2}- H^{2}\right).\notag
\end{align}
\end {lemma}

\begin{corollary}
Under the conditions of Theorem \ref{main result for powerF flow},
\begin{align}\label{ev.inequa.f}
\partial_{t} f\leq &\Delta_{\dot \Phi} f
+\frac{(n+1)}{n H^{n}}\left\langle\nabla
 f,\nabla H^{n}\right\rangle_{\dot \Phi}
-\frac{(n-1)}{n K}\left\langle\nabla  f,\nabla K\right\rangle_{\dot \Phi}
-\frac{ H^{n}}{n K}\left|\nabla f\right|_{\dot \Phi}^2
\\
&- C_{16}{ H}f,\notag
\end{align}
where $C_{16}=\delta C^{*}$
\end{corollary}
\begin{proof}
For concave $F$, applying the similar argument as in Theorem \ref{pinching},
Lemma \ref{a key estimate2},
inequality \eqref{concave pinched} and Lemma \ref{important inequality} to \eqref{evtildef}
gives the conclusion.
For convex $F$, we only instead use the curvature pinching inequality \eqref{convex pinched}
and note that $\frac{K}{F^{n}} \leq \frac{n^{n}K}{H^{n}}$.
\end{proof}

Now we can show exponential convergence to the sphere.
\begin{proposition}\label{convergence of f}
Under the conditions of Theorem \ref{main result for powerF flow}, the rate of
convergence of $f$ to $0$ as $t\rightarrow \infty$ is exponential.
\end{proposition}
\begin{proof}
For concave $F$, applying the maximal principal to \eqref{ev.inequa.f},
by using the fact $F \leq \frac{H}{n}$ and Proposition \ref{lower speed bound} gives
\begin{align}
\partial_{t} f_{\max}(t)\leq -C_{17}f_{\max}(t).\notag
\end{align}
which implies that
\begin{align}
f_{\max}(t)\leq C_{18}\e^{- C_{17}t},\notag
\end{align}
where $C_{18}=f_{\max}(0)$.
For convex $F$, we only instead use that $H \geq n \varepsilon_{0}F$ and Proposition \ref{lower speed bound}.
This proves the
Proposition.
\end{proof}

Now if $t_0$ can be taken big enough so that $M_{t}$ can be
represented as $\graph u$ for $[t_0, \infty)$, then since all the
derivatives of $u$ are uniformly bounded independent of time, applying
Arzel-Ascoli Theorem we conclude that the graph $u(t, \cdot \,)$ is defined
on $[t_0, \infty)$ and converges to a unique function $u_{\infty}$.
By standard argument as in\cite[page 467]{CR-M09} this implies that $X_\infty$ is an immersion,
 and since the convergence is smooth we
can assure that $X_\infty$  must be a compact
embedded hypersurface.
On the other hand, Proposition \ref{the shape of M} says that all
points on $X_\infty$ are umbilic points. In conclusion, the only
possibility is that $\mathcal {S}$ represents a sphere in
$\mathbb{R}^{n+1}$ and, by the volume-preserving
properties of the flow, such a sphere has to enclose the same volume
as the initial condition $X_0(M)$.

Finally, from Proposition \ref{convergence of f} we can conclude
with the standard arguments as in \cite[Theorem 3.5]{Sch06},
\cite[Theorem 7.3]{CS10} that the
flow converges exponentially to the sphere $\mathcal {S}$
in $\mathbb{R}^{n+1}$ in the $C^{\infty}$-topology.


\bigskip
\noindent{\bf  Acknowledgments.}
\quad Some of this research was carried out while I was a post-doctor fellow
at the School of Mathematics at Sichuan University.
I am grateful to the School for providing a fantastic research atmosphere.
I would like to express special thanks to Professor An-Min Li for his enthusiasm and encouragement,
Professor Guanghan Li for his patience, suggestion and helpful discussion on this topic,
and Professor Haizhong Li for his interest. The research is partially supported by China Postdoctoral Science Foundation Grant 2015M582546
and Natural Science Foundation Grant 2016J01672 of the Fujian Province, China.

\bibliographystyle{spbasic}

\end{document}